\documentclass[11pt]{ectaart2}

\RequirePackage[OT1]{fontenc} \RequirePackage{amsthm}
\RequirePackage[cmex10]{amsmath} %\RequirePackage{natbib}
\RequirePackage{amsfonts} \RequirePackage{amssymb} \RequirePackage{a4wide}
\RequirePackage{graphicx,xcolor}
\RequirePackage[normalem]{ulem}

\usepackage{xcolor}

\startlocaldefs \numberwithin{equation}{section}
\theoremstyle{plain}

\newtheorem{theorem}{Theorem}[section]

\newtheorem{remark}{Remark}[section]
\newtheorem{lemma}{Lemma}[section]
\newtheorem{lemma_def}{Lemma and Definition}[section]

\endlocaldefs

\def\@bysame#1{\vrule height 1.5pt depth -1pt width 3em \hskip
0.5em\relax}

\newcommand{\N}{ \mathbb{N} }

\newcommand{\Z}{ \mathbb{Z} }

\newcommand{\R}{ \mathbb{R} }

\newcommand{\trunc}[1]{ {\lfloor #1 \rfloor} }
\newcommand{\wh}[1]{ \widehat{ #1 } }
\newcommand{\wt}[1]{ \widetilde{ #1 } }

\newcommand{\calA}{\mathcal{A}}
\newcommand{\calB}{\mathcal{B}}
\newcommand{\calC}{\mathcal{C}}
\newcommand{\calD}{\mathcal{D}}
\newcommand{\calE}{\mathcal{E}}
\newcommand{\calF}{\mathcal{F}}
\newcommand{\calG}{\mathcal{G}}

\newcommand{\calN}{\mathcal{N}}
\newcommand{\calO}{\mathcal{O}}

\newcommand{\calT}{\mathcal{T}}

\newcommand{\calW}{\mathcal{W}}

\newcommand{\calY}{\mathcal{Y}}

\newcommand{\eins}{\mathbf{1}}
\newcommand{\matA}{{ \mathbf A}}
\newcommand{\matB}{{\mathbf B}}

\newcommand{\matC}{{\mathbf C}}

\newcommand{\matP}{{ \mathbf P}}

\newcommand{\vecf}{{ \mathbf f}}
\newcommand{\vece}{{ \mathbf e}}

\newcommand{\vecv}{{ \mathbf v}}

\newcommand{\vecw}{{ \mathbf w}}
\newcommand{\vecx}{{ \mathbf x}}

\newcommand{\vecY}{{ \mathbf Y}}

\newcommand{\bfgamma}{\boldsymbol{\gamma}}

\newcommand{\bfSigma}{\boldsymbol{\Sigma}}

\newcommand{\Var}{{\mbox{Var\,}}}
\newcommand{\Cov}{{\mbox{Cov\,}}}

\newcommand{\matid}{I}

\newcommand{\trace}{ \operatorname{tr} }

\newcommand{\cadlag}{c\`adl\`ag }

\newcommand{\chg}[2]{#2}

\endlocaldefs

\begin{document}

\begin{frontmatter}

% "Title of the paper"
%of Changes in the Mean and Scale
\title{Asymptotics for High--Dimensional Covariance Matrices and Quadratic Forms with Applications to the Trace Functional and Shrinkage}
\runtitle{Inference for the trace}

% indicate corresponding author with \corref{}
% \author{\fnms{John} \snm{Smith}\corref{}\ead[label=e1]{smith@foo.com}\thanksref{t1}}
% \thankstext{t1}{Thanks to somebody}
% \address{line 1\\ line 2\\ printead{e1}}
% \affiliation{Some University}

\begin{aug}
\author{\fnms{Ansgar} \snm{Steland${}^\ast$}%\thanksref{t1}
\ead[label=e1]{steland@stochastik.rwth-aachen.de}}
\author{\fnms{Rainer} \snm{von Sachs} %\thanksref{t3}
\ead[label=e2]{rainer.vonsachs@uclouvain.be}}
%\author{\fnms{Third} \snm{Author}
%\ead[label=e3]{third@somewhere.com}
\ead[label=u1,url]{www.stochastik.rwth-aachen.de}

%\thankstext{t1}{corresponding author}
%\thankstext{t2}{First supporter of the project}
%\thankstext{t3}{Second supporter of the project}
%\runauthor{A. Steland, R. von Sachs}

\address{Institute of Statistics\\ RWTH Aachen University \\ W\"ullnerstr. 3, D-52056 Aachen, Germany\\
\printead{e1}} %\printaddresses \bigskip
\address{Institut de Statistique, Biostatistique et Sciences Actuarielles (ISBA)\\ Universit\'e catholique de Louvain\\ Voie du Roman Pays 20, B-1348 Louvain-la-Neuve, Belgium\\
\printead{e2}} \printaddresses {July 10, 2017} \bigskip

%\address{Address of the Third author, usually few lines long,
%usually few lines long\\
%\printead{e3,u1}}
\end{aug}

%\author{\fnms{Ansgar} \snm{Steland}\ead[label=e1]{steland@stochastik.rwth-aachen.de}}
%\address{\printead{e1}}
  %\and
  %\author{\fnms{???} \snm{???}\ead[label=e2]{???}}
%\address{Institute of Statistics, W\"ullnerstr. 3, D-52056 Aachen\\ \printead{e2}}
%\affiliation{RWTH Aachen University}

\runauthor{A. Steland, R. von Sachs}

\begin{abstract}
	We establish large sample approximations for an arbitray number of bilinear forms of the sample variance-covariance matrix of a high-dimensional vector time series using $ \ell_1 $-bounded weighting vectors. Estimation of the asymptotic covariance structure is also discussed. The results hold true without any constraint on the dimension, the number of forms and the sample size or their ratios. 
	Concrete and potential applications are widespread and cover high-dimensional data science problems such as projections onto sparse principal components or more general spanning sets as frequently considered, e.g. in classification and dictionary learning. 
    As two specific applications of our results, we study in greater detail the asymptotics of the  trace functional and shrinkage estimation of the covariance matrices. In shrinkage estimation, it turns out that the asymptotics differs for weighting vectors bounded away from orthogonaliy and nearly orthogonal ones in the sense that their inner product converges to $0$.
\end{abstract}


%\begin{keyword}[class=JEL]
%\kwd[Primary ]{?,?}
%\kwd{}
%\kwd[; secondary ]{62E20}
%\end{keyword}


\end{frontmatter}

\section{Introduction}

A large number of procedures studied to analyze high--dimensional vector time series of dimension $ d_n $ depending on the sample size $n$ relies on projections, e.g. by projecting the observed random vector onto a spanning set of a lower dimensional subspace of dimension $ L_n  $. Examples include sparse principal component analysis, see e.g. \cite{TibshiraniWittenHastie2009}, in order to reduce dimensionality of data, sparse portfolio replication and index tracking as studied by \cite{BrodieEtAl2009}, or dictionary learning, see \cite{BarronCohenDahmenDeVore2008}, where one aims at representing input data by a sparse linear combination of the elements of a dictionary, frequently obtained as the union of several bases and/or historical data.

When studying projections, it is natural to study the associated bilinear form $ \vecv_n' \wh{\bfSigma}_n \vecw_n $, $ \vecv_n, \vecw_n \in \R^{d_n} $, representing the dependence structure in terms of the projections' covariances. Here and throughout the paper $ \wh{\bfSigma}_n $ is the (uncentered) sample variance--covariance matrix. In order to conduct inference, large sample distributional approximations are needed. For a vector time series model given by correlated linear processes, we established in \cite{StelandvonSachs2016} a strong approximation by a Brownian motion for a single quadratic form, provided the weighting vectors are uniformly bounded in the $ \ell_1 $-norm. It turned out that the result does not require any condition on the ratio of the dimension and the sample size, contrary to many asymptotic results in highdimensional statistics and probability.

In the present article, we study the more general case of an {\em increasing} number of quadratic forms as arising when projecting onto a sequence of subspaces whose dimension converges to $ \infty $. Noting that the analysis of autocovariances of a stationary linear time series appears as a special case of our approach, there are a few recent results related to our work: \cite{WuMin2005} established a central limit theorem for a finite number of autocovariances, whereas in \cite{WuHuangZheng2010} the case of long memory series has been studied. \cite{Jirak2012} has studied the asymptotic theory for detecting a change in mean of a vector time series with growing dimension.

To treat the case of an increasing number of bilinear forms, we consider two related but different frameworks: The first framework uses a sequence of Euclidean spaces $ \R^{d_n} $ equipped with the usual Euclidean norm. The second framework embeds those spaces in the sequence space $ \ell_2 $ equipped with the $ \ell_2 $--norm.  It is shown that, in both frameworks, an increasing number of, say $L_n$, quadratic forms can be approximated by Brownian motions without any constraints on $ L_n $, $ d_n $ and $n$ apart from $ n \to \infty $. One of our main results asserts that, for the assumed time series models, one can define, on a new probability space, equivalent versions and a Gaussian process $ \calG_n $ taking values in $ C([0,1], \R^{L_n} ) $, such that
\[
  \sup_{t \in [0,1]} \frac{1}{\sqrt{nL_n}} \left| \left( \vecv_n^{(j)}{}'(\wh{\bfSigma}_{\trunc{nt}}- E \wh{\bfSigma}_{\trunc{nt}} )\vecw_n^{(j)} \right)_{j=1}^{L_n} - \calG_n(t)   \right| = o_P(1),
\]
as $ n \to \infty $, almost surely (a.s.), without any constraints on $ L_n $, $d_n $.

We believe that those results have many applications in diverse areas, as indicated above.
In this paper we study in some detail two direct applications: The first application considers the trace operator, which equals the trace matrix norm $ \| \cdot \|_{tr} $ when applied to covariance matrices. We show that the trace of the sample covariance matrix, appropriately centered, can be approximated by a Brownian motion, a.s. on a new probability space, which also establishes the convergence rate
\[
  \left| \| \wh{\bfSigma}_n \|_{tr} - \| E \wh{\bfSigma}_n \|_{tr} \right| = O_P( n^{-1/2} d_n ).
\]

The second application elaborated in this paper is shrinkage estimation of a covariance matrix as studied in depth for i.i.d. sequences of high--dimensional random vectors as well as dependent vector time series, see by \cite{LedoitWolf2003}, \cite{LedoitWolf2004} and \cite{Sancetta2008} amongst others. In order to regularize the sample variance-covariance matrix, the shrinkage estimator considers a convex combination with a well--defined target that usually corresponds to a simple regular model. We consider the identity target, i.e. a multiple of the identity matrix $ \matid_n $ of dimension $d_n$. To the best of our knowledge, large sample approximations for those estimators have not yet been studied. We show that, uniformly in the shrinkage weight for the convex combination, a bilinear form given by the shrinkage estimator can be approximated by a Gaussian process when it is centered at the shrunken true covariance matrix using the same shrinkage weight. By uniformity, the result also holds for the widely used estimator of the optimal shrinkage weight. For this estimated optimal weight the convergence rate under quite general conditions is known. It turns out that, when comparing the matrices in terms of a natural pseudodistance induced by bilinear forms, the convergence rate carries over from the optimal weight's estimator. We also compare the shrinkage estimator using the estimated optimal weight with an oracle estimator using the unknown optimal weight. Last, we study the case of nearly (i.e. asymptotically) orthogonal vectors. As a consequence of the Kabatjanskii-Levenstein bound, see \cite{Tao2013}, this property allows to place much more unit vectors on the unit sphere. It turns out that for nearly orthogonal vectors 
the nonparametric part dominates in large samples, contrary to the situation for vectors bounded away from orthogonality.

The high-dimensional time series model of the paper is as follows. At time $n$, $ n \in \N $, we observe a $ d = d_n $ dimensional mean zero vector time series
\[
  \vecY_{ni} = (Y_{ni}^{(1)}, \dots, Y_{ni}^{(d_n)} )',\qquad 1 \le i \le n,
\]
defined on a common probability space $ ( \Omega, \calF, P ) $,  whose coordinates are causal linear processes,
\begin{equation}
\label{ModelLinearProcess}
Y_{ni}^{(\nu)} = \sum_{j=0}^\infty c_{nj}^{(\nu)} \epsilon_{i-j}, \qquad i \in \N_0, \nu = 1, \dots, d_n,
\end{equation}
where $ \{ \epsilon_t \} $ are independent mean zero error terms\chg{with $ E( \epsilon_k^3 ) = 0 $ for all $k $}{},possibly not identically distributed, such that $ E | \epsilon_k |^{4+\delta} < \infty $ for some $ \delta > 0 $ and $ n^{-1} \sum_{i=1}^n E(\epsilon_i^r) $, $ r = 2, 3, 4 $, converges. The coefficients $ c_{nj}^{(\nu)}  $ may depend on $n$ and are therefore also allowed to depend on the dimension $d_n$. We impose the following growth condition.

\textbf{Assumption (A):} The coefficients $ c_{nj}^{(\nu)} $ of the linear processes (\ref{ModelLinearProcess})
satisfy 
\begin{equation}
\label{SuffCond}
\sup_{n \in \N} \max_{1 \le \nu \le d_n} | c_{nj}^{(\nu)} |^2 \ll j^{-3/2-\theta}
\end{equation}
for some $ 0 < \theta < 1/2 $. 

It is well known that Assumption (A) covers common classes of weakly dependent time series such as ARMA($p,q$)-models as well as a wide range of long memory processes. We refer to \cite{StelandvonSachs2016} for a discussion.

%Our assumptions on the coefficients of the linear processes (see below) are weak enough to cover many frequently used time series models such as ARMA models of known orders and even a wide class of long-range dependent time series. 

Define the (centered) bilinear form 
\begin{equation}
\label{DefQForm}
  Q_n(\vecv_n, \vecw_n ) =  \sqrt{n} \vecv_n' (\wh{\bfSigma}_n - \bfSigma_n) \vecw_n, \qquad \vecv_n, \vecw_n \in \R^{d_n},
\end{equation}
where
\begin{equation}
\label{DefSigmaHat}
  \wh{\bfSigma}_n = \frac1n \sum_{i=1}^n \vecY_{ni} \vecY_{ni}'
  \qquad \text{and} \qquad
  \bfSigma_n = E \wh{\bfSigma}_n.
\end{equation}
The class of proper (sequences of) weighting vectors,
\[
  \vecw_n = (w_{n1}, \dots, w_{nd_n} )', \qquad n \ge 1,
\]
studied throughout the paper is the set $ \calW $ of those sequences $ \{ \vecw_n : n \ge 1 \} $, $ \vecw_n \in \R^{d_n} $, $n \ge 1 $, which have uniformly bounded $ \ell_1 $-norm in the sense that
\begin{equation}
\label{l1Condition}
  \sup_{n \in \N} \| \vecw_n \|_{\ell_1} = \sup_{n \in \N} \sum_{\nu=1}^{d_n} | w_{n\nu} | < \infty.
\end{equation}
$ \ell_1$-weighting vectors naturally arise in various applications
such as sparse principal component analysis as, see e.g. \cite{TibshiraniWittenHastie2009}, or sparse financial portfolio selection as studied by \cite{BrodieEtAl2009}. For a more detailed discussion we refer to \cite{StelandvonSachs2016}.

\chg{}{
	It is worth mentioning that our results easily carry over to weighting vectors with uniformly bounded $ \ell_2 $-norm, provided one relies on standardized versions of the bilinear form (\ref{DefQForm}). First notice that
	conditions (\ref{SuffCond}) and (\ref{l1Condition}) allow us to control the linear process coefficients of a projected time series,
	$ \vecw_n' \vecY_{ni} $, $ 1 \le i \le n $, which are then $ O( j^{-3/2-\theta}) $ and therefore decay at the same rate as the original time series. The $ \ell_2 $ assumption 
  \begin{equation}
  \label{l2Condition}
	   \sup_{n \ge 1} \sqrt{ \sum_{\nu=1}^{d_n} w_{n\nu}^2 } < \infty,
   \end{equation}
   leads to the estimate
\[
  \left|  \sum_{\nu=1}^{d_n} w_{n\nu} c_{nj}^{(\nu)} \right| = O \left( \sqrt{ \sum_{\nu=1}^{d_n} | c_{nj}^{(\nu)} |^2 } \right)
\]
For bounded dimension, (\ref{SuffCond}) yields the estimate $ j^{-3/2-\theta} $ for the latter expression, but for growing dimension this does not hold in general. Assuming 
$ 
	\sum_{\nu=1}^{\infty} | c_{nj}^{(\nu)} |^2 < \infty
$
for all $ j \ge 0 $, is, however, not  reasonable for a high-dimensional setting, since then $ c_{nj}^{(\nu)} = o(1) $, as $ \nu \to \infty $. For example, if
$ c_{nj}^{(\nu)} = \rho_{\nu}^j $, for $ 0 < | \rho_{\nu} | < 1 $, $ \nu \ge 1 $, the latter assumption would rule out the case of observing $ d_n $ autoregressive time series of order $1$ with autoregressive parameters bounded away from zero. On the other hand, if $ w_{n \nu} \ge w_{\min} > 0 $, for $ \nu \ge 1 $, and $  c_{nj}^{(\nu)} \ge c_{j,\min} > 0 $ for $ \nu \ge 1 $ and $ n \ge 1 $, then 
$ \sum_{\nu=1}^{d_n} w_{n\nu} c_{nj}^{(\nu)} \ge d_n w_{\min} c_{j,\min}  \to \infty $ for each lag $ j \ge 0 $. This can be fixed by considering
$ \wt{\vecw}_n' \vecY_{ni} $ instead of $ \vecw_n' \vecY_{ni} $, where $ \wt{\vecw}_{n} = d_n^{-1} \vecw_n $. Then
\[
  \left| \sum_{\nu=1}^{d_n} \wt{w}_{n\nu} c_{nj}^{(\nu)} \right| = O\left( \frac{1}{d_n} \sum_{\nu=1}^{d_n} |c_{nj}^{(\nu)}|^2 \right) \ll j^{-3/2 - \theta}, \qquad j \ge 0.
\]
Next observe that by Jensen's inequality
\[
  \sum_{\nu=1}^{d_n} | \wt{w}_{n\nu} | = \frac{1}{d_n} \sum_{\nu=1}^{d_n} \sqrt{w_{n\nu}^2}
  \le \frac{1}{\sqrt{d_n}} \sqrt{ \sum_{\nu=1}^\infty w_{n\nu}^2}.
\]
Hence,  (\ref{SuffCond})  and the $\ell_2$-condition (\ref{l2Condition}) imply  $ \{ \wt{\vecw}_n : n \ge 1 \} \in \calW $ and  $ \wt{\vecw}_n'\vecY_{ni} $ is a linear time series with coefficients decaying at the same rate $ j^{-3/2-\theta} $ as the original time series. Clearly, for any sequences $ \vecv_n $, $ \vecw_n $ of weighting vectors with uniformly bounded $ \ell_2$-norm, we have the scaling property
\[
  Q_n( \vecv_n, \vecw_n ) = d_n^2 Q_n( \wt{\vecv}_n, \wt{\vecw}_n), 
\]
where $ \wt{\vecv}_n = d_n^{-1} \vecv_n $ and $ \wt{\vecw}_n = d_n^{-1} \vecw_n $ have uniformly bounded $ \ell_1$-norm. But if one standardizes $ Q_n $, the factor $ d_n^2 $ cancels. Hence, in this sense several of our theoretical results can be also applied to study projection onto vectors with uniformly bounded  $ \ell_2 $-norm.
}

The rest of the paper is organized as follows. In Section~\ref{Sec Asymptotics QF}, we introduce the partial sums and partial sum processes associated to an increasing number of bilinear forms and establish the strong and weak approximation theorems for those bilinear forms. The application to the trace functional is discussed in Section~\ref{Sec Asymptotics Trace}. The large sample approximations for shrinkage estimators of covariance matrices are studied in depth in Section~\ref{Sec Asymptotics Shrinkage}.

\section{Large sample approximations for high--dimensional bilinear forms}
\label{Sec Asymptotics QF}

\subsection{Definitions and review}

Let us define the matrix--valued partial sums
\begin{align}
\label{DefPSum1}
\wh{\bfSigma}_{nk} & = \left( \sum_{i=1}^k Y_i^{(\nu)} Y_i^{(\mu)} \right)_{1 \le \nu, \mu \le d_n}, \\
\label{DefPSum2}
\bfSigma_{nk} & = \left( \sum_{i=1}^k E Y_i^{(\nu)} Y_i^{(\mu)} \right)_{1 \le \nu, \mu \le d_n},
\end{align}
for $ n, k \ge 1 $, and put 
\begin{equation}
\label{DefDNK}
D_{nk}( \vecv_n, \vecw_n) = \vecv_n'( \wh{\bfSigma}_{nk} - \bfSigma_{nk} ) \vecw_n,
\qquad n, k \ge 1,
\end{equation}
for two sequences of weighting vectors $ \{ \vecv_n \}, \{ \vecw_n \} \in \calW $. The associated \cadlag processes will be denoted by
\begin{equation}
\label{DefCadlagDNK}
\calD_n(t;\vecv_n, \vecw_n) = n^{-1/2} D_{n,\trunc{nt}}( \vecv_n, \vecw_n) = \vecv_n' n^{-1/2}( \wh{\bfSigma}_{n, \trunc{nt}} - \bfSigma_{n, \trunc{nt}} ) \vecw_n, \qquad t \in [0,1], n \ge 1.
\end{equation}
Especially, we have
\begin{equation}
\label{DefDNCLTVersion}
\calD_n(1) =  \calD_n(1; \vecv_n, \vecw_n ) = \vecv_n' \sqrt{n}( \wh{\bfSigma}_n - \bfSigma_n ) \vecw_n,  n \ge 1,
\end{equation}
with
$
\bfSigma_n = E \wh{\bfSigma}_n = \frac{1}{n} \sum_{i=1}^n E( \vecY_{ni} \vecY_{ni} )'.
$
%If $ \{ \vecY_{ni} : 1 \le i \le n \} $ is stationary, then $ \bfSigma_n $ simplifies to  $ \bfSigma_n = E( \vecY_{n1} \vecY_{n1}') $.

For some sequence of standard Brownian motions,  $ \{ B_n(t) : t \in [0,\infty) \} $, $ n \ge 1 $, and any constant $ N>0 $ we can introduce the rescaled version $ \{ N^{-1/2} B_N( t N ) : s \in [0,1] \} $ called the {\em [0,1]--version of $ B_n $}. In \cite{StelandvonSachs2016} the following result on the asymptotics of a single bilinear form for a uniformly bounded $ \ell_1 $-projections is shown: 

\begin{theorem}
	\label{Th_Basic} 
	Suppose $ \vecY_{ni} $,  $1 \le i \le n $, $ n \ge 1 $, is a vector time series according to model (\ref{ModelLinearProcess}) that satisfies Assumption (A). Let $ \vecv_n $ and $ \vecw_n $ be weighting vectors with uniformly bounded $ \ell_1 $--norm in the sense of (\ref{l1Condition}). Then, for each $n \in \N $, there exists  equivalent versions of $ D_{nk}( \vecv_n, \vecw_n) $ and $ D_n(t;\vecv_n, \vecw_n) $, $ t \ge 0 $, again denoted by $ D_{nk}(\vecv_n, \vecw_n) $ and $ D_n(t; \vecv_n, \vecw_n) $, and a standard Brownian motion $ \{ W_n(t) : t \ge 0 \} $, which depends on $ (\vecv_n, \vecw_n) $, i.e. $ W_n(t) = W_n(t;\vecv_n, \vecw_n) $, both defined on some probability space $ (\Omega_n, \mathcal{F}_n, P_n ) $, such that for some $ \lambda > 0 $ and a constant $ C_n $
	\begin{equation}
	\label{StrongApprox0}
	| D_{nt}(\vecv_n, \vecw_n ) - \alpha_n(\vecv_n, \vecw_n)  W_n(t) | \le C_n t^{1/2-\lambda}, \qquad \text{for all $ t > 0 $ a.s.,}
	\end{equation}
	where $ \alpha_n(\vecv_n, \vecw_n) $ is defined in (\ref{AlphaN}). If $ C_n n^{-\lambda} = o(1) $, as $ n, t \to \infty $, this implies the strong approximation
	\begin{equation}
	\label{StrongApprox}
	\sup_{t \in [0,1]} | \calD_n(t; \vecv_n, \vecw_n) - \alpha_n(\vecv_n, \vecw_n) W_n(\trunc{nt}/n) | =o(1),
	\qquad \text{a.s.},
	\end{equation}
	as $ n \to \infty $, for the $ [0,1]$--version of $ W_n $ as well as the CLT
	\begin{equation}
	\label{MainCLT}
	| \calD_n(1;\vecv_n, \vecw_n) - \alpha_n(\vecv_n, \vecw_n) W_n(1) | = o(1), \qquad \text{a.s.},
	\end{equation}
	as $ n \to \infty $, i.e. $ \calD_n(1) $ is asymptotically $ \calN(0, \alpha_n^2 ) $.
\end{theorem}

A multivariate version for  $ L \in \N $ bilinear forms, which approximates 
\[
  \left( \calD_n( \cdot; \vecv_n^{(j)}, \vecw_n^{(j)} ) \right)_{j=1}^{L}
\] 
by a $ L $--dimensional Brownian motion, has been shown in \cite[Th.~4.2]{StelandvonSachs2016}. This result allows to consider the dependence structure which arises when mapping $ \vecY_n = \vecY_{nn} $ onto the subspace $ \mathcal{P}_L = \text{span}\{ \vecw_n^{(j)} : j = 1, \dots, L \} $ spanned by $ L $ weighting vectors,
\[
\vecw_n^{(j)} \in \calW, \qquad j = 1, \dots, L.
\]
We have the canonical mapping, called {\em projection (onto $ \mathcal{P}_L $)} in the sequel, 
\begin{equation}
\label{ScaledProjYn}
\vecY_n \mapsto \matP_n \vecY_n, \qquad
\matP_n = ( \vecw_n^{(1)}, \dots, \vecw_n^{(L)} )',
\end{equation}
which represents the orthogonal projection onto $ \mathcal{P}_L $, if
$ \vecw_n^{(1)}, \dots, \vecw_n^{(L)} $ are orthonormal. 
The associated  variance-covariance matrix is
\[
\Cov( \matP_n \vecY_n ) =
\left( \Cov( \vecw_n^{(j)} \vecY_n, \vecw_n^{(k)} \vecY_n ) \right)_{1 \le j, k \le L} = \left( \vecw_n^{(j)}{}' \bfSigma_n  \vecw_n^{(k)} \right)_{1 \le j, k \le L}.
\]
If the $ \vecw_n^{(j)} $'s are eigenvectors of $ \bfSigma_n $, then
$ \Cov( \matP_n \vecY_n ) $ is a diagonal matrix. But that property
is lost for more general spanning vectors.
Given the sample $ \vecY_{n1}, \dots, \vecY_{nn} $ of $ d_n $-dimensional random vectors, the canonical
nonparametric statistical estimators of   $ \bfSigma_n $ and $  \Cov( \matP_n \vecY_n )  $ are $ \wh{\bfSigma}_n $ as defined in (\ref{DefSigmaHat})  and
\[
\wh{\Cov}( \matP_n \vecY_n ) = 
\left(
\vecw_n^{(j)}{}' \wh{\bfSigma}_n \vecw_n^{(k)} 
\right)_{1 \le j, k \le L}.
\]
The entries of $ \wh{\Cov}( \matP_n \vecY_n ) $ consist of bilinear forms as studied in Theorem~\ref{Th_Basic}, and for fixed $L$ its multivariate extension suffices to study the dependence structure of the projection onto $ \mathcal{P}_L $. This no longer holds, if $ L $ is allowed to grow as the sample size increases, i.e. when studying the case $ L = L_n  \to \infty $, as $ n \to \infty $. Indeed, the treatment of that high--dimensional situation is much more involved. As we shall see, it requires a different scaling and a more involved mathematical framework: The strong approximations we establish in this paper take place in the Euclidean space $ \R^{L_n} $ of growing dimension and the infinite-dimensional Hilbert space $ \ell_2 $, respectively.

Thus, to go beyond the case of a finite number of bilinear forms we now consider 
\[
Q_n( \vecv_n^{(j)}, \vecw_n^{(j)} ) = \vecv_n^{(j)}{}' \wh{\bfSigma}_n \vecw_n^{(j)}, \qquad j = 1, \dots, L_n,
\]
where $ ( \{ \vecv_\ell^{(j)} \}_{\ell=1}^{\infty}, \{ \vecw_\ell^{(j)} \}_{\ell=1}^\infty ) \in \calW \times \calW $, $ j = 1, \dots, L_n $, are $L_n $ pairs of uniformly $ \ell_1 $-bounded sequences of weighting vectors and $L_n$ may tend to infinity as $ n \to \infty $. We are interested in the joint asymptotics of the centered and scaled versions of the corresponding statistics (\ref{DefDNCLTVersion}) given by
\begin{equation}
\label{DefDNJ}
\calD_{nj} = 
L_n^{-1/2} \calD_n(1; \vecv_n^{(j)}, \vecw_n^{(j)} )
%  = \sqrt{n} L_n^{-1} \vecv_n^{(j)}{}' (\wh{\bfSigma}_n - \bfSigma_n) \vecw_n^{(j)}
, \qquad j = 1, \dots, L_n,
\end{equation}
and the associated sequential \cadlag processes
\begin{equation}
\label{DefDNSJ}
\calD_{nj}(t) = L_n^{-1/2} \calD_n(t; \vecv_n^{(j)}, \vecw_n^{(j)} ), \qquad t \in [0,1], \qquad j = 1, \dots, L_n,
\end{equation}
cf. (\ref{DefCadlagDNK}). The additional factor $ L_n^{-1/2} $ anticipates the right scaling to obtain a large sample approximation.

Further, we are interested in studying weighted averages where averaging takes place over all $L_n$ forms and all sample sizes $n \in \N $. Let $ \lambda_n $ be the weight for sample size $n$ and $ \mu_\rho $ the weight for the $\rho$th quadratic form associated to a pair of sequences of weighting vectors $ \{ \vecv_n^{(\rho)}, \vecw_n^{(\rho)} : n \ge 1 \} $ for $ \rho = 1, \dots, L_n $, $ n \ge 1 $.  Define for $ k \ge 1 $
\begin{equation}
\label{DefQF}
D_k( \{ (\vecv_n^{(\rho)}, \vecw_n^{(\rho)} ) \} )
= \sum_{n,m=1}^\infty \lambda_n \lambda_m \sum_{\rho=1}^{L_n} \sum_{\sigma=1}^{L_m} \mu_\rho \mu_\sigma
\vecv_n^{(\rho)}{}' ( \wh{\bfSigma}_{nmk} - \bfSigma_{nmk} ) \vecw_m^{(\sigma)}, 
\end{equation}
where 
\begin{align}
\label{DefSNMK1}
\wh{\bfSigma}_{nmk} &= \sum_{i \le k} \vecY_{ni} \vecY_{mi}', \\
\label{DefSNMK2}
\bfSigma_{nmk} & = E( \wh{\bfSigma}_{nmk} ),
\end{align}
for $ n, m \ge 1 $. Notice the relations
\[
  \wh{\bfSigma}_{nnk} = \wh{\bfSigma}_{nk}, \qquad \wh{\bfSigma}_{nn} = n \wh{\bfSigma}_n
\]
between (\ref{DefSigmaHat}), (\ref{DefPSum1}) and (\ref{DefSNMK1}). 
$ D_k( \{ (\vecv_n^{(\rho)}, \vecw_n^{(\rho)} ) \} ) $ depends on all 
weights $ \{ v_{n\nu}^{(\rho)}, w_{n\nu}^{(\rho)}, \lambda_n, \mu_\rho : 1 \le \nu \le d_n, 1 \le \rho \le L_n, n \ge 1 \} $
but is measurable w.r.t $ \calG_k = \sigma( \vecY_{ni} : n \in \N, i \le k ) $. Now, for any sample size $ M $ we may consider the
associated \cadlag process associated to (\ref{DefQF})
\begin{equation}
\calD_{M}( t; \{ (\vecv_n^{(\rho)}, \vecw_n^{(\rho)} ) \} ) = M^{-1/2} D_{\trunc{Mt}}( \{ (\vecv_n^{(\rho)}, \vecw_n^{(\rho)} ) \} ),
\qquad t \in [0, 1].
\end{equation}

\subsection{Preliminaries}

Before proceeding, recall the following facts on the Hilbert space $ \ell_2 $ and strong approximations in Hilbert spaces. \chg{The space $ \ell_2 $ of sequences $ f = (f_j)_j $ with $ \sum_j f_j^2 < \infty $ is a separable Hilbert space when equipped with the inner product $ (f, g) = \sum_j f_j g_j $, $ f = (f_j)_j, g = (g_j)_j \in \ell_2 $, and the induced norm $ \| f \|_{\ell} = \sqrt{(f,f)} $. 
}{We shall denote the inner product of an arbitray Hilbert space by $ < \cdot, \cdot >$, the induced norm by$ | \cdot | $ and the operator semi-norm of an operator $ T : H \to H $ by $ \| T \|_{op} = \sup_{x \in H: |x|=1} | Tx | $. Our results take place in the Hilbert space $ \ell_2 $ of all sequences $ f = (f_j)_j $ with $ \sum_j f_j^2 < \infty $, which is a separable Hilbert space when equipped with the inner product $ (f, g) = \sum_j f_j g_j $, $ f = (f_j)_j, g = (g_j)_j \in \ell_2 $, and the induced norm $ \| f \|_{\ell_2} = \sqrt{(f,f)} $. The associated operator norm of an operator $ T : \ell_2 \to \ell_2 $ is simply denoted by $ \| T \| $. For two random variables $ X, Y $ defined on $ (\Omega, \calF, P) $ with $ E(X^2), E(Y^2) < \infty $, we denote the inner product by $ (X,Y)_{L_2} $, where $ L_2 = L_2(\Omega, \calF, P) $.}

Sufficient conditions for a strong approximation of partial sums of dependent random elements taking values in a separable Hilbert space require the control of the associated conditional covariance operator. Denote the underlying
probability space by $ ( \Omega, \calF, P ) $. Let $ X = (X_j)_j $ be a random element defined on $ (\Omega, \calF, P ) $ taking values in $ \ell_2 $ with $ E \| X \|_{\ell_2}^2 = \sum_j E(X_j^2) < \infty $. The covariance operator $ C_X : \ell_2 \to \ell_2 $ associated to $ X $ is defined by 
\[
  C_X( f ) = E[ (f,X) X ] = \left( \sum_j f_j E( X_j X_k ) \right)_k, \qquad f = (f_j)_j \in \ell_2.
\]
To any $ \sigma $-field $ \calA $ we may associate the conditional covariance operator of $X$ given $ \calA $,
\begin{equation}
\label{DefCondCovOp}
  C_X(f|\calA) = E[ (f,X) X | \calA ] = \left( \sum_j f_j E( X_j X_k | \calA ) \right)_k, \qquad f = (f_j)_j \in \ell_2.
\end{equation}
Covariance operators are symmetric positive linear operators with operator norm
\begin{equation}
\label{DefOpNorm}
   \| C_X( \cdot | \calA ) \| = \sup_{f \in \ell_2: \| f \|_{\ell_2} = 1 } \| ( C_X(f), f ) \|_{\ell_2}.
\end{equation}
For further properties and discussion see \cite{Bosq2000}. 

A strong invariance principle in $ \ell_2 $ deals with the a.s. approximation of partial sums of $ \ell_2 $--valued random elements by a Brownian motion in $ \ell_2 $. Recall that a random element $ B = \{ B(t) : t \in [0,1] \} $ with values in $ C([0,1], \ell_2 ) $ is called Brownian motion in $ \ell_2 $, if
\begin{itemize}
  \item[(i)] $ B(0) = 0 \in \ell_2 $,
  \item[(ii)] for all $ 0 \le t_1 < \cdots < t_n \le 1 $ the increments $ B(t_{i+1}) - B(t_i) $, $ i = 1, \dots, n-1 $, are
  independent and
  \item[(iii)] for all $ 0 \le s, t \le 1 $ the increment $ B(t)-B(s) $ is Gaussian with mean $0$ and covariance operator
  $  \min(s,t) K $ for some nonnegative linear and self-adjoint operator $K$ on $ \ell_2 $ such that
  $ \sum_{i=1}^\infty ( K e_i, e_i ) < \infty $, where $ \{ e_i \} $ is some orthonormal system for $ \ell_2 $.
\end{itemize}
If $ K = C_X $ for some random element $X$, $B$ is the Brownian motion generated by $X$.  The definition for a general separable Hilbert space is analogous.

A strong invariance principle or strong approximation for a sequence $ \zeta_1, \zeta_2, \dots $ of random elements taking values in an arbitrary separable Hilbert space  $H$ with inner product $ < \cdot, \cdot > $ and induced norm $ | \cdot | $ asserts that they can be redefined on a rich enough probability space such that there exists
  a Brownian motion $ B $ with values in $H$ and covariance operator $ C_\zeta $ such that, a.s.,
  \begin{equation}
  \label{SIP1}
    \left| \sum_{j \le t} \zeta_j - B(t) \right|
     \le c t^{1/2-\lambda},
  \end{equation}
  for  constants $ \lambda > 0 $ and $ c < \infty $, if the dimension of $H$ is finite, and,
  \begin{equation}
  \label{SIP2}
      \left| \sum_{j \le t} \zeta_j - B(t) \right| = o( \sqrt{t \log( \log t ) } ),
  \end{equation}
  as $ t \to \infty $, a.s., if $H$ is infinite dimensional.

Throughout the paper we write, for two arrays $ \{ a_{n',m'} \} $ and $ \{ b_{n',m'} \} $ of real numbers, $ a_{n',m'} \stackrel{n',m'}{\ll} b_{n',m'} $, if there exists a constant $ c $ such that  $ a_{n',m'} \le c b_{n',m'} $ for all $ n', m' $.

\subsection{Large sample approximations}

We aim at showing a strong approximation for the $ D([0,1]; \ell_2) $-valued processes 
\[
  \mathcal{D}_n = ( \mathcal{D}_{nj} )_{j = 1}^{L_n},  \qquad n \ge 1,
\]
where the coordinate processes $ \calD_{nj} $ are given by
\[
  \mathcal{D}_{nj}(t) = \frac{1}{\sqrt{n L_n}} D_{n, \trunc{nt}}( \vecv_n^{(j)}, \vecw_n^{(j)} ), \qquad t \in [0,1], \quad n \ge 1,
\]
with $ D_{nk}( \vecv_n^{(j)}, \vecw_n^{(j)} ) = \vecv_n^{(j)}{}'( \wh{\bfSigma}_{nk} - \bfSigma_{nk} ) \vecw_n^{(j)} $
for $ j = 1, \dots, L_n $, $n \ge 1 $, cf. (\ref{DefDNK})  and (\ref{DefDNJ}). The above processes can be expressed as partial sums.

\begin{lemma}
	\label{LemmaRepr}
We have the representation
\begin{equation}
\label{ReprSum}
  D_{nk} = \sum_{i=1}^k \xi_i^{(n)},
\end{equation}
for $ k = 1, \dots, n $, $ n \ge 1 $, leading to
\begin{equation}
  \mathcal{D}_{n}(t) = \frac{1}{\sqrt{n L_n}} \sum_{i=1}^{ \trunc{nt} } \xi_i^{(n)}, \qquad t \in [0,1], n \ge 1,
\end{equation}
where the random elements $  \xi_i^{(n)} $ are defined in (\ref{DefDLN}).
\end{lemma}

To introduce the conditional covariance operators associated to $ \calD_n $, denote the filtration $ \calF_m = \sigma( \epsilon_i : i \le m ) $, $ m \in \Z $, and define
\[
  \calC^{(n)}( f | \calF_0 ) = E[ (f, \mathcal{D}_n(1) ) \mathcal{D}_n(1) \, | \, \calF_0 ],
  \qquad f \in \ell_2, n \ge 1.
\]
Let us also introduce the unconditional covariance operator
\[
  \calC^{(n)}( f ) = E[ (f, Z_n ) Z_n ], \qquad f \in \ell_2,
\]
where $ Z_n = ( Z_{nj} )_{j = 1}^{L_n} $, $ n \ge 1 $, with random variables $ Z_{nj} $, $ 1 \le j \le L_n $, satisfying
 $ E( Z_{nj} ) = 0 $ and 
\[ 
  E( Z_{n j} Z_{n k} ) = L_n^{-1} \beta_n^2(j, k),
\] 
for $ j, k = 1, \dots, L_n $. Here 
\begin{equation}
\label{BetaJK}
  \beta_n^2(j,k) = \beta_n^2( \vecv_n^{(j)}, \vecw_n^{(j)}, \vecv_n^{(k)}, \vecw_n^{(k)} )
\end{equation}
are the quantities introduced in (\ref{BetaNuMu}), the asymptotic covariance parameters of the bilinear forms corresponding to the
pairs $ ( \vecv_n^{(j)}, \vecw_n^{(j)} ) $ and $ (\vecv_n^{(k)}, \vecw_n^{(k)}) $, $ 1 \le j,k \le L_n $.

The following technical but crucial result establishes the convergence of $ \calC^{(n)}( \cdot | \calF_0 ) - \calC^{(n)}( \cdot ) $ in the operator semi--norm in expectation and provides us with a convergence rate.

\begin{theorem} 
\label{ThConvCovOp}
Suppose $ (\vecv_n^{(j)}, \vecw_n^{(j)} ) $, $ j = 1, \dots, L_n $,
have uniformly bounded $ \ell_1 $-norms,
\[
  \sup_{n \ge 1} \max_{1 \le j \le L_n} \max\{ \| \vecv_n^{(j)} \|_{\ell_1}, \| \vecw_n^{(j)} \|_{\ell_1} \} \le C < \infty,
\]
for some constant $ C $. Let 
\begin{equation}
\label{DefSnm}
 S_{n', m'}^{(n)} = \frac{1}{\sqrt{n' L_n}} \sum_{k=m'+1}^{m'+n'} \xi_k^{(n)}, \qquad
 m', n' \ge 1,  n \ge 1,
\end{equation} 
with $ \xi_k^{(n)} $ defined by (\ref{DefDLN}). Define
\[
  \calC^{(n)}_{n',m'}( f | \calF_{m'} )
  = E[ (f, S_{n', m'}^{(n)} ) S_{n', m'}^{(n)} | \calF_{m'} ],
\]
for $ f \in \ell_2 $. Then
\[
  E \| \calC^{(n)}_{n',m'}( \cdot | \calF_m ) 
   - \calC^{(n)}( \cdot ) \| 
   \stackrel{n',m'}{\ll} (n')^{-\theta/2},
\]
where $ \| \cdot \| $ denotes the operator norm defined in (\ref{DefOpNorm}).
\end{theorem}

We are now in a position to formulate the first main result on the large sample approximations of $ L_n $ bilinear forms when $ L_n $ converges to infinity, in terms of the $ \ell_1 $-- as well as the $ \ell_2 $--norm. The  results holds true under the weak assumption that the weighting vectors have uniformly bounded $ \ell_1 $ norm.

%For an approximation with respect to the $ \ell_1 $--norm, an additional scaling with $ L_n^{-1} $ is required to compensate the fact that small approximation errors in the coordinates have a larger effect on the $ \ell_1$-norm than on the $ \ell_2 $-norm.

\begin{theorem} 
\label{Th_Many_QFs}
Let $ \vecY_{ni} $, $ 1 \le i \le n$, be a vector time series following model (\ref{ModelLinearProcess}) and satisfying Assumption (A).  Suppose that $ ( \{ \vecv_\ell^{(1)} \}_{\ell=1}^\infty, \{ \vecw_\ell^{(1)} \}_{\ell=1}^\infty ), \dots, ( \{ \vecv_n^{(L_n)} \}_{\ell=1}^\infty, \{ \vecw_n^{(L_n)} \}_{\ell=1}^\infty )  \in \calW \times \calW $, $ n \ge 1 $, have uniformly bounded $ \ell_1 $-norm, i.e.
\[
  \sup_{n \ge 1} \max_{1 \le j \le L_n} \max\{ \| \vecv_n^{(j)} \|_{\ell_1}, \| \vecw_n^{(j)} \|_{\ell_1} \} \le C,
\]
for some constant $ C < \infty $.
Then all processes can be redefined on a rich enough probability space, such that there exists, for each $n$, a Brownian motion of dimension $ L_n $,
\[ 
  B_n(t) = B_n(t; \{ ( \vecv_n^{(j)}, \vecw_n^{(j)} ) : 1 \le j \le L_n \} ), \qquad t \ge 0,
\] 
with coordinates $ B_n(t)_j $, $ j = 1, \dots, L_n $, and covariance function given by
\begin{equation}
\label{CovariancesBB}
  E( B_n(s)_j B_n(t)_k ) = \min(s,t) L_n^{-1} \beta_n^2(j,k),
\end{equation}
for $ 1 \le j, k \le L_n $ and $ s, t \ge 0 $, such that the following assertions hold true.
\begin{itemize}
  \item[(i)]  In the Euclidean space $ \R^{L_n} $ we have the strong approximation
  \begin{equation}
  \label{StrongApprox Ln}
    \| L_n^{-1/2} D_{nt} - B_n(t) \|_{\R^{L_n}} =
    \left\|  \sum_{i=1}^t L_n^{-1/2} \xi_i^{(n)} - B_n(t) \right\|_{\R^{L_n}} \le C_n t^{-1/2-\lambda_n},
  \end{equation}
  a.s, for  constants $ C_n < \infty $ and $ \lambda_n > 0 $, where $ \lambda_n $ depends only on $ L_n $, $ \delta $ and $ \theta $. 
\end{itemize}
Provided
\begin{equation}
\label{CrucialCond}
  C_n n^{-\lambda_n} = o(1),
\end{equation}
as $ n \to \infty $, the following assertions hold.
\begin{itemize}
\item[(ii)] With respect to the $ \ell_2 $--norm we have
\[
  \sup_{t \in [0,1]} \sum_{j=1}^{L_n} \left| \calD_{nj}(t) - \overline{B}_n\left( \frac{ \trunc{nt} }{ n } \right)_j \right|^2 
  = o(1)
%  \le C_n n^{-2 \lambda},
\]
as $ n \to \infty $, a.s., for the $ [0,1] $--version $ \overline{B}_n $ of $ B_n $.
\item[(iii)] With respect to the $ \ell_1 $--norm we have
\[
 \sup_{t \in [0,1]} \frac{1}{\sqrt{L_n}} \sum_{j=1}^{L_n} \left| 
  \calD_{nj}(t) -  \overline{B}_n\left( \frac{\trunc{nt}}{n} \right)_j \right| 
%  \le C_n n^{-\lambda},
  = o(1),
\]
as $n \to \infty $, a.s., for the $ [0,1] $--version $ \overline{B}_n $ of $ B_n $, and with respect
to the maximum norm
\[
  \sup_{t \in [0,1]} \max_{j \le L_n} \left| \calD_{nj}(t) - \overline{B}_n\left( \frac{ \trunc{nt} }{ n } \right)_j \right| = o(1),
\]
as $ n \to \infty $, a.s..
\item[(iv)]
  Let $ \{ \lambda_n : n \in \N \} $ and $ \{ \mu_\rho : \rho \in \N \} $ 
  be $\ell_1 $--weights. Then there exist constants $ \lambda > 0 $ and $C < \infty $ and $ \alpha(  \{ \lambda_n \}_{n=1}^\infty, \{ \mu_ \rho \}_{\rho=1}^\infty, \{ (\vecv_n^{(\rho)}, \vecw_n^{(\rho)} ) \}_{n,\rho=1}^\infty ) \ge 0 $, such that for equivalent versions and a standard Brownian motion $B$ on $ [0, \infty) $, defined on a new probability space, 
  \begin{equation}
  \label{WSIP1}
    | D_t( \{ (\vecv_n^{(\rho)}, \vecw_n^{(\rho)} ) \} ) - 
    \alpha(  \{ \lambda_n \}_{n=1}^\infty, \{ \mu_ \rho \}_{\rho=1}^\infty, \{ (\vecv_n^{(\rho)}, \vecw_n^{(\rho)} ) \}_{n,\rho=1}^\infty ) B(t) |
    \le C t^{1/2-\lambda},
  \end{equation}
  a.s., for all $ t > 0 $. Further, for any sample size $M$
  \begin{equation}
  \label{WSIP2}
  \sup_{t \in [0,1]} | \calD_M( t;  \{ (\vecv_n^{(\rho)}, \vecw_n^{(\rho)} ) \} )
  - \alpha(  \{ \lambda_n \}_{n=1}^\infty, \{ \mu_ \rho \}_{\rho=1}^\infty, \{ (\vecv_n^{(\rho)}, \vecw_n^{(\rho)} ) \}_{n,\rho=1}^\infty )  \overline{B}_M(t) |
  \le C M^{-\lambda},
  \end{equation} 
  a.s., for the $ [0,1] $--version $ \overline{B}_M$ of $B$.
\end{itemize}
\end{theorem}

\begin{remark}
\label{RemarkBBConstr}
The Brownian motions can be constructed such that
\begin{equation}
\label{PropertyBB}
  \sup_{t \in [0,1]} \left| \overline{B}_n\left( \frac{ \trunc{nt} }{n} \right)_j  -  L_n^{-1/2} \alpha_n( \vecv_n^{(j)}, \vecw_n^{(j)} ) W_n\left( \frac{ \trunc{nt}}{n} ; \vecv_n^{(j)}, \vecw_n^{(j)} \right) \right| = o(1),
\end{equation}
as $ n \to \infty $, a.s., if (\ref{CrucialCond}) holds, where $ W_n(\cdot;  \vecv_n^{(j)}, \vecw_n^{(j)} ) $ is as in Theorem~\ref{Th_Basic}, for $ j = 1, \dots, L_n $.
\end{remark}

%\begin{remark}  
%We conjecture that the constant $ C_n $ in (\ref{CrucialCond}) can be chosen universally, such that the strong approximations in (i) and (ii) can be strengthened to $ O( n^{-\lambda} ) $ as it is the case in assertion (iii). 
%\end{remark}

Due to assertion (iv) of the above theorem we may conjecture that (\ref{CrucialCond}) holds, cf. the discussion in \cite{StelandvonSachs2016}, but we have neither a proof nor a counterexample. The following result studies the relevant processes in the infinite--dimensional space $ \ell_2 $ and yields an approximation in probability taking into account the additional factor $ \log \log( n ) $.

\begin{theorem} 
\label{WeakApprox_ManyQFs}	
Suppose that the assumptions of Theorem~\ref{Th_Many_QFs} hold. In the Hilbert space $ \ell_2 $ we have the strong approximation
		  \begin{equation}
		  \label{StrongApprox ell_2} 
		  \| L_n^{-1/2} D_{nt} - B_n(t) \|_{\ell_2} =
		  \left\|  \sum_{k=1}^t L_n^{-1/2} \xi_k^{(n)} - B_n(t) \right\|_{\ell_2} = o( \sqrt{ t \log( \log t ) } ),   
		  \end{equation}
		  as $ t \to \infty $, a.s. There exists a sequence $ \{ \delta_n : n \ge 1 \} \subset \N $ such that with $ N = \lceil n \log \log(n) \rceil $
		  \begin{equation} 
		  \label{StrongApprox_HSpace 1}
		  \max_{\delta_n \le k \le n}  \left\| \frac{1}{\sqrt{N L_n}} \sum_{i=1}^k \xi_i^{(n)} - \overline{B}_N\left( \frac{k}{N} \right) \right\|_{\ell_2}
		  = o_P(1), 
		  \end{equation} 
		  $ n \to \infty $,  for the $ [0,1] $--version $ \overline{B}_N(t) = N^{-1/2}B_n( t N ) $. In other words,
		  \begin{equation}
		  \label{StrongApprox_HSpace 2}
		  \max_{\delta_n \le k \le n}  \left\| \sqrt{ \frac{n}{N} } \mathcal{D}_n \left( \frac{k}{n} \right) - \overline{B}_N\left( \frac{k}{N} \right) \right\|_{\ell_2}
		  = o_P(1),
		  \end{equation}
		  or equivalently
		  \begin{equation}
		  \label{StrongApprox_HSpace 3}
		  \sup_{ \frac{\delta_n}{n} \le t \le 1 }
		  \left\|
		  \sqrt{ \frac{n}{N} } \mathcal{D}_n(t) - \overline{B}_N\left( \frac{ \trunc{nt} }{ N } \right) 
		  \right\|_{\ell_2}
		  = o_P(1),
		  \end{equation}
		  as $ n \to \infty $.
\end{theorem}

The above result eliminates the condition (\ref{CrucialCond}), but we have no detailed information about the sequence $ \delta_n $. 

\chg{}{The question arises, whether the above results are limited to linear processes. As the main arguments deal with approximating martingales, we have the following result, which suggests that the class of vector time series to which the main results of this paper apply is larger.}

\chg{}{
\begin{theorem}
	 Let $ ( \{ \vecv_\ell^{(1)} \}_{\ell=1}^\infty, \{ \vecw_\ell^{(1)} \}_{\ell=1}^\infty ), \dots, ( \{ \vecv_n^{(L_n)} \}_{\ell=1}^\infty, \{ \vecw_n^{(L_n)} \}_{\ell=1}^\infty )  \in \calW \times \calW $, $ n \ge 1 $, be projection vectors with uniformly bounded $ \ell_1 $-norm, i.e.
	\[
	\sup_{n \ge 1} \max_{1 \le j \le L_n} \max\{ \| \vecv_n^{(j)} \|_{\ell_1}, \| \vecw_n^{(j)} \|_{\ell_1} \} \le C,
	\]
	for some constant $ C < \infty $. Let  $ \vecY_{ni} $, $ 1 \le i \le n $, be a $ d_n $-dimensional vector time series such that $ D_{nk}( \vecv_n^{(\nu)}, \vecw_n^{(\mu)} ) $ can be approximated by the martingales $ M_k^{(n)}( \vecv_n^{(\nu)}, \vecw_n^{(\mu)} ) $, $ k \ge 1 $, defined in (\ref{DefMartingales}) in $ L_2 $ with rate $ n^{-\theta} $,  $ \nu = 1, \dots, d_n $, for certain sequences of coefficients, $ c_{nj}^{(\nu)} $, $ j \ge 0 $, $ \nu  = 1, \dots, d_n $, satisfying Assumption (A) and a sequence of independent mean zero random variables $ \{ \epsilon_k : k \ge 0 \} $ with $ \sup_{k \ge 1} E ( \epsilon_k^4 ) < \infty $. If
	\[
	  \sup_{n \ge 1} \sup_{k \ge 1} \max_{1 \le j \le L_n} \max \{ E | \vecv_n^{(j)}{}' \vecY_{nk} |^{4+2\delta}, E | \vecw_n^{(j)}{}' \vecY_{nk} |^{4+2\delta} \} < \infty,
	\] 
	for some $ \delta > 0 $, then the results of this section still hold true. 
\end{theorem}
}

\subsection{Proofs}

\begin{proof}[Proof of Lemma~\ref{LemmaRepr}] 
	We argue as in the proof of Theorem~\ref{Th_Basic} given in  \cite{StelandvonSachs2016}, where
	it was shown that the partial sum (\ref{DefDNK}) associated to a single bilinear form $ Q( \vecv_n, \vecw_n ) $ attains the representation
	\begin{equation}
	\label{DefDnk}
	%  \frac{1}{ \sqrt{n} } 
	D_{nk}( \vecv_n, \vecw_n ) 
	=
	% \frac{1}{ \sqrt{n} } 
	\sum_{i \le k} \xi_i^{(n)}( \vecv_n, \vecw_n ) 
	\end{equation}
	with Gaussian random variables 
	\begin{equation}
	\label{DefXi_1}
	\xi_i^{(n)}( \vecv_n, \vecw_n ) = Y_{ni}( \vecv_n ) Y_{ni}( \vecw_n ) - E( Y_{ni}( \vecv_n ) Y_{ni}( \vecw_n ) ), 
	\qquad i = 1, 2, \dots, n, \ n \ge 1,
	\end{equation}
	for linear processes
	\begin{align*}
	Y_{ni}{(\vecv_n)} &  = \sum_{j=0}^\infty c_{nj}^{(v)} \epsilon_{i-j}, \qquad
	Y_{ni}{(\vecw_n)} 
	= \sum_{j=0}^\infty c_{nj}^{(w)} \epsilon_{i-j}, \qquad i \in \Z, n \in \N,
	\end{align*}
	with coefficients
	\begin{equation}
	\label{DefCNWeights}
	c_{nj}^{(v)} = \sum_{\nu=1}^{d_n} v_\nu c_{nj}^{(\nu)},
	\qquad 
	c_{nj}^{(w)} = \sum_{\nu=1}^{d_n} w_\nu c_{nj}^{(\nu)},
	\end{equation}
	for $ j \ge 0 $ and $ n \ge 1 $. For $ L_n $ pairs of weighting vectors $ (\vecv_n^{(j)}, \vecw_n^{(j)} ) $, $ j = 1, \dots, L_n $, we consider 
	the corresponding partial sum process where the summands are the $ L_n $-dimensional vectors
	\begin{equation}
	\label{DefDLN}
	\xi_i^{(n)} = \left(  \xi_i^{(n)}( j ) \right)_{j=1}^{L_n}, \qquad
	\xi_i^{(n)}( j )  = \xi_i^{(n)}( \vecv_n^{(j)}, \vecw_n^{(j)} ), \ j = 1, \dots, L_n, 
	\end{equation}
	for $ i = 1, 2, \dots $, which we, however, also interpret as random elements taking values in $ \ell_2 $. This completes
	the proof.
\end{proof}

  \cite[Th. 1]{Philipp1986} asserts that (\ref{SIP1}) and (\ref{SIP2}), respectively, hold, if the following conditions for the scaled partial sums $ S_{n'}(m') = (n')^{-1/2} \sum_{k=m'+1}^{m'+n'} \zeta_k $, $ k, m', n' \ge 0 $,  are satisfied:
  \begin{itemize}
  	\item[(I)] $ \sup_{j \ge 1} E | \zeta_j |^{2+\delta} < \infty $ for some $ \delta > 0 $.
  	\item[(II)] For some $ \varepsilon > 0 $
  	\[
  	E | E( S_{n'}(m')) | \calF_{m'} ) | \stackrel{n',m'}{\ll} (n')^{-\varepsilon}.
  	\]
  	\item[(III)] There exists a covariance operator $ C $ such that the conditional
  	covariance operators
  	\[
  	C_{n'}(f|\calF_{m'}) = E[ (f, S_{n'}(m') ) S_{n'}(m') | \calF_{m'} ), \qquad f \in H,
  	\]
  	converge in the operator semi-norm $ \| \cdot \| $  to $ C(f) $ in expectation with rate $ (n')^{-\theta} $, i.e.
  	\[
  	E \| C_{n'}(\cdot| \calF_{m'} ) - C(\cdot) \| \stackrel{n',m'}{\ll} (n')^{-\theta},
  	\]
  	for some $ \theta > 0 $.
  \end{itemize}
  
  \chg{}{For a discussion of this result and extensions see \cite{Zhang2004}.} As shown by \cite{DehlingPhilipp1982}, 
  the strong invariance principle (\ref{SIP2}) also holds true for strictly stationary sequences taking values in a separable
  Hilbert space, which possess a finite moment of order $ 2+\delta' $, $ \delta' > 0 $, and are strong mixing with mixing coefficients
  satisfying $ \alpha(k) = O( k^{-(1+\varepsilon)(1+2/\delta')} ) $, for some $ \varepsilon > 0 $. The above conditions are, however, more
  convenient when studying linear processes. \cite{Wu2007} has studied strong invariance principles for a univariate nonlinear time series using the physical dependence measure, which is easy to verify for linear processes. Extensions to vector-valued time series (of fixed dimension) have been provided by \cite{LiuLin2009}. We rely on the conditions of \cite{Philipp1986}, since they allow to study time series of growing dimension and taking values in the infinite-dimensional space $ \ell_2 $ in a relatively straightforward way.

As a preparation for the proof of Theorem~\ref{ThConvCovOp}, we need the following lemma dealing with the uniform convergence of unconditional and conditional covariances of the approximating martingales defined by
\begin{equation}
\label{DefMartingales}
M_{m'}^{(n)}( \nu ) = \sum_{k=0}^{m'} \left\{
\wt{f}^{(n)}_{0,0}( \nu )( \epsilon_k^2 - \sigma_k^2 )
+ \epsilon_k \sum_{l=1}^\infty \wt{f}_{l,0}^{(n)}( \nu ) \epsilon_{k-l}
\right\}, \qquad m' \ge 0,
\end{equation}
where for brevity $ \wt{f}_{l,i}^{(n)}( \nu ) = \wt{f}_{l,i}^{(n)}( \vecv_n^{( \nu )}, \vecw_n^{(\nu)} ) $, $ l, i = 0, 1, 2, \dots $, $ \nu = 1, \dots, d_n $.

\begin{lemma}  
	\label{HLemmaCov}
	Under Assumption (A) we have
	\begin{equation}
	\label{ConvCovCondF_m}
	\sup_{n \ge 1}  \left\|  \sup_{1 \le \nu, \mu}  \left| E[ ( M_{m'+n'}^{(n)}( \nu ) - M_{m'}^{(n)}( \nu ) )
	( M_{m'+n'}^{(n)}( \mu ) - M_{m'}^{(n)}( \mu ) ) | \calF_{m'} ] - n' \beta_n(\nu, \mu) \right| \right\|_{L_1} \stackrel{n',m'}{<<} (n')^{1-\theta},
	\end{equation}
	which implies
	\begin{equation}
	\label{CovCondBounded}
	\sup_{n \ge 1} \sup_{n', m' \ge 1}
	(n')^{-1} \sup_{1 \le \nu} E[ ( M_{m'+n'}^{(n)}( \nu ) - M_{m'}^{(n)}( \nu ) )^2 | \calF_{m'} ] < \infty,
	\end{equation}
	a.s. Further,
	\begin{equation}
	\label{ConvCovUncond}
	(n')^{-1}  \sup_{n \ge 1} \sup_{1 \le \nu, \mu} 
	\left|
	E( M_{m'+n'}^{(n)}( \nu ) - M_{m'}^{(n)}( \nu ) )
	( M_{m'+n'}^{(n)}( \mu ) - M_{m'}^{(n)}( \mu ) ) - n' \beta_n( \nu, \mu )
	\right| \stackrel{n',m'}{<<} (n')^{-\theta},
	\end{equation}
	which implies
	\begin{equation}
	\label{CovUncondBounded}
	\sup_{n \ge 1} \sup_{n', m' \ge 1}
	(n')^{-1} \sup_{1 \le \nu} E[ ( M_{m'+n'}^{(n)}( \nu ) - M_{m'}^{(n)}( \nu ) )^2 ] < \infty,
	\end{equation}
\end{lemma}
\begin{proof}
	A direct calculation leads to
	\begin{align*}
	C_{Mn}( \nu, \mu ) 
	& = E[ (M_{m'+n'}^{(n)}(\nu) - M_{m'}^{(n)}(\nu) )( M_{m'+n'}^{(n)}( \mu ) - M_{m'}^{(n)}(\mu) ) \, | \, \calF_{m'} ] \\
	& = C_{Mn}^{(0)}( \nu, \mu ) +  \cdots +  C_{Mn}^{(3)}( \nu, \mu ),
	\end{align*}
	where
	\begin{align*}
	C_{Mn}^{(0)}( \nu, \mu ) & = \sum_{k,k'=m'+1}^{m'+n'} \sum_{l=1}^\infty \wt{f}^{(n)}_{l,0}( \nu ) \wt{f}_{\ell,0}^{(n)}(\mu) E[ \epsilon_k \epsilon_{k-l}( \epsilon_{k'}^2 - \sigma_{k}^2 ) | \calF_{m'} ], \\ 
	C_{Mn}^{(1)}( \nu, \mu ) & = \sum_{k=m'+1}^{m'+n'} \wt{f}_{0,0}^{(n)}( \nu ) \wt{f}_{0,0}^{(n)}( \mu ) ( \gamma_k - \sigma_k^4 ), \\ 
	C_{Mn}^{(2)}( \nu, \mu ) & = \sum_{k=m'+1}^{m'+n'} \sum_{l=0}^{k-m'+1} \wt{f}_{l,0}^{(n)}( \nu ) \wt{f}_{l,0}^{(n)}( \mu ) \sigma_k^2 \sigma_{k-l}^2, \\
	C_{Mn}^{(3)}( \nu, \mu ) &= \sum_{k=m'+1}^{m'+n'} \sum_{l,l'=k-m'}^\infty \wt{f}_{l,0}^{(n)}( \nu ) \wt{f}_{l',0}^{(n)}( \mu ) \sigma_k^2 \epsilon_{k-l} \epsilon_{k-l'},
	\end{align*}
	for $ m', n' \ge 0 $. Let us first estimate $ C_{Mn}^{(1)}( \nu, \mu ) + C_{Mn}^{(2)}( \nu, \mu )  $. We have
	\[
	\sup_{n \ge 1}   \sup_{1 \le \nu, \mu} | C_{Mn}^{(1)}( \nu, \mu ) + C_{Mn}^{(2)}( \nu, \mu ) - n'  \beta_n(\nu, \mu) | \stackrel{n',m'}{\ll} (n')^{1-\theta},
	\]
	see (\ref{Ass5}). Next, we show that
	\begin{equation}
	\label{CMn3_toshow}
	\sup_{n \ge 1}  E \left[ \sup_{1 \le \nu, \mu} | C_{Mn}^{(3)}( \nu, \mu ) |  \right] \stackrel{n',m'}{\ll} (n')^{1-\theta}.
	\end{equation}
	Recall that $ c_2 = \sup_{k \ge 1} E( \epsilon_k^2) < \infty $ and assume w.l.o.g. $c_2 = 1 $ in what follows.
	The Cauchy--Schwarz inequality yields
	\begin{align*}
	&  | (n')^{-1} C_{Mn}^{(3)}( \nu, \mu ) | \\
	& \quad \le
	\frac{1}{n'} \sum_{k=m'+1}^{m'+n'} \left| \sum_{l=k-m'}^\infty \wt{f}_{l,0}^{(n)}( \nu )  \epsilon_{k-l} \right|
	\left| \sum_{l'=k-m'}^\infty \wt{f}_{l',0}^{(n)}( \mu )  \sigma_k^2  \epsilon_{k-l'} \right| \\
	& \quad \le 
	\sqrt{
		\frac{1}{n'} \sum_{k=m'+1}^{m'+n'} \left( \sum_{l=k-m'}^\infty | \wt{f}_{l,0}^{(n)}( \nu ) | \epsilon_{k-l} \right)^2 
	} 
	\sqrt{
		\frac{1}{n'} \sum_{k=m'+1}^{m'+n'} \left( \sum_{l=k-m'}^\infty | \wt{f}_{l,0}^{(n)}( \mu ) | \epsilon_{k-l'} \right)^2 
	}. 
	\end{align*}
	Using $ \| \cdot \|_{\ell_1} \le \| \cdot \|_{\ell_4} $ and Jensen's inequality, we obtain
	\begin{align*}
	\sqrt{
		\frac{1}{n'} \sum_{k=m'+1}^{m'+n'} \left( \sum_{l=k-m'}^\infty | \wt{f}_{l,0}^{(n)}( \nu ) | \epsilon_{k-l} \right)^2 
	} 
	&  \le  
	\sqrt{
		\frac{1}{n'} \sum_{k=m'+1}^{m'+n'} \left( \sum_{l=k-m'}^\infty [ \wt{f}_{l,0}^{(n)}( \nu ) ]^4 \epsilon_{k-l}^4 \right)^{1/2} 
	}  \\
	& \le 
	\sqrt{
		\frac{1}{n'} \sum_{k=m'+1}^{m'+n'} \left( \sum_{l=k-m'}^\infty \sup_{1 \le \nu} [ \wt{f}_{l,0}^{(n)}( \nu ) ]^4 \epsilon_{k-l}^4 \right)^{1/2} 
	} \\
	& \le 
	\left(
	\frac{1}{n'} \sum_{k=m'+1}^{m'+n'}  \sum_{l=k-m'}^\infty \sup_{1 \le \nu} [ \wt{f}_{l,0}^{(n)}( \nu ) ]^4 \epsilon_{k-l}^4 
	\right)^{1/4},
	\end{align*}
	where the upper bound does not depend on $ \nu $. Hence,
	\begin{align*}
	E  \sup_{1 \le \nu, \mu} | (n')^{-1} C_{Mn}^{(3)}( \nu, \mu ) |
	&  \le E   \sqrt{ 
		\frac{1}{n'} \sum_{k=m'+1}^{m'+n'}  \sum_{l=k-m'}^\infty \sup_{1 \le \nu} [ \wt{f}_{l,0}^{(n)}( \nu ) ]^4 \epsilon_{k-l}^4 
	} \\
	& \le \sqrt{ 
		\frac{1}{n'} \sum_{k=m'+1}^{m'+n'} \sum_{l=k-m'}^\infty \sup_{1 \le \nu} [ \wt{f}_{l,0}^{(n)}( \nu ) ]^4 \gamma_{k-l} 
	}
	\end{align*} 
	Using $ | \wt{f}_{l,0}^{(n)}( \nu ) | = O( \| \vecv_n^{(\nu)} \|_{\ell_1} \| \vecv_n^{(\mu)} \|_{\ell_1} l^{-3/4-\theta/2} ) $, uniformly in $n$, $ \sup_{k \ge 1} \gamma_k < \infty $ and the elementary fact that $ \sum_{k=m'}^{m'+n'} \sum_{l=k-m'+1}^\infty l^{-3-2\theta} = O( (n')^{1-2\theta} ) $\chg{, (\ref{ConvCovCondF_m}) follows.}{,  (\ref{CMn3_toshow}) follows.} 
	\chg{}{Lastly, consider $ C_{Mn}^{(0)}( \nu, \mu ) $. Since the indices satisfy $ k-\ell \le m' $, $ \epsilon_{k-\ell} $ is $ \calF_{m'} $-measurable, whereas $ \epsilon_k, \epsilon_{k'}$ are independent from $ \calF_{m'}$. Hence
	\[
	  C_{Mn}^{(0)}( \nu, \mu ) = \sum_{k,k'=m'+1}^{m'+n'} \sum_{l=1}^\infty \wt{f}_{l,0}^{(n)}( \nu ) \wt{f}_{l,0}^{(n)}( \mu ) 
	  \epsilon_{k-\ell} E( \epsilon_k ( \epsilon_{k'}^2 - \sigma_{k'}^2 ) ).
	\]
	Clearly, for $ k \not= k' $ the summands vanish, such that 
	\[
	  C_{Mn}^{(0)}( \nu, \mu ) = \sum_{k=m'+1}^{m'+n'} \sum_{l=1}^\infty \wt{f}_{l,0}^{(n)}( \nu ) \wt{f}_{l,0}^{(n)}( \mu ) 
	    \epsilon_{k-l} E( \epsilon_k^3 ).
	\]
	If $ E( \epsilon_k^3 ) = 0 $ for all $k$, $ C_{Mn}^{(0)}( \nu, \mu ) = 0 $. Otherwise, put $ c_3 = \sup_{k \ge 1 } E | \epsilon_k |^3 $. We have the estimate
	\begin{align*}
	    (n')^{-1} E | C_{Mn}^{(0)}( \nu, \mu ) |
	      & \le c_3 \frac{1}{n'} \sum_{k=m'}^{m'+n'} \sum_{l=1}^{\infty} |  \wt{f}_{l,0}^{(n)}( \nu )  \wt{f}_{l,0}^{(n)}( \mu ) | E |\epsilon_{k-l} |    \\
	      & \le c_3 \frac{1}{n'} \sum_{k=m'}^{m'+n'}  \sqrt{ \sum_{l=1}^\infty [\wt{f}_{l,0}^{(n)}( \nu )]^2 [\wt{f}_{l,0}^{(n)}( \mu )]^2  } \\
	      & =  O\left( \frac{1}{n'} \sum_{k=m'}^{m'+n'} \sqrt{ \sum_{l=1}^{\infty} l^{-3-2\theta} } \right) 
	      \stackrel{n',m'}{\ll} (n')^{-\theta}.
	\end{align*}
	Hence (\ref{ConvCovCondF_m}) follows,
	} 
	The above arguments also imply that 
	\[ \sup_{n \ge 1} \sup_{1 \le \nu, \mu} E (n')^{-1}  | C_{Mn}( \nu, \mu ) | \stackrel{n',m'}{\ll} 1, \] 
	\chg{}{since 
	\[
	    |(n')^{-1} C_{Mn}(\nu, \mu)| \le (n')^{-1}  |C_{Mn}(\nu, \mu) - n' \beta_n(\nu,\mu) | + (n')^{-1} | \beta_n(\nu,\mu) |, 
	\]
	where the first term is a.s. finite, since its $L_1$-norm is $ \ll (n')^{-\theta} $, and the second one is $ \ll 1 $,
	uniformly in $ 1 \le \nu ,\mu $, such that
	\[
	  \sup_{n \ge 1} \sup_{1 \le \nu, \mu} E |(n')^{-1} C_{Mn}(\nu, \mu)| \le (n')^{-1}  |C_{Mn}(\nu, \mu) - n' \beta_n(\nu,\mu) | + (n')^{-1} | \beta_n(\nu,\mu) | \ll 1.
	\]
}
	which in turn implies (\ref{CovCondBounded}). To verify (\ref{ConvCovUncond}) one first conditions on $ \calF_{m'} $ and then argues similarly in order to estimate $ E C_{Mn}^{(3)}( \nu, \mu ) $. Observe that with 
	$ c_{24} = \max \{ \sup_{k \ge 1} (\sigma_k^2)^3, \sup_{k \ge 1} \gamma_k^4  \sup_{k \ge 1} \sigma_k^2\} $
	\begin{align*}
	E (n')^{-1} C_{Mn}^{(3)}(\nu, \mu) 
	& \le c_{24} \frac{1}{n'} \sum_{k=m'+1}^{m'+n'} \sum_{l=0}^\infty \left( \sup_{1 \le \nu} | f_{l,0}^{(n)}( \nu ) | \right)^2 \\
	& \le c_{24} \sqrt{ \frac{1}{n'} \sum_{k=m'+1}^{m'+n'} \sum_{l=0}^\infty \left( \sup_{1 \le \nu} | f_{l,0}^{(n)}( \nu ) | \right)^4  } \\
	& \stackrel{n',m'}{\ll} (n')^{-\theta},
	\end{align*}
	using $ \| \cdot \|_{\ell_1} \le \| \cdot \|_{\ell_4} $ and Jensen's inequality,
	which verifies (\ref{ConvCovUncond}) and in turn (\ref{CovUncondBounded}).
\end{proof}

Introduce for $ m', n' \ge 1 $ and each coordinate $ 1 \le \nu \le L_n $ the partial sums 
$ D_{n',m'}^{(n)}( \vecv_n^{(\nu)}, \vecw_n^{(\nu)} ) = \sum_{i=m'+1}^{m'+n'} \xi_i^{(n)}( \vecv_n^{(\nu)}, \vecw_n^{(\nu)} ) $ and denote the appropriately scaled versions by
\begin{equation}
\label{DefDScaled}
T_{n',m'}^{(n)}(\nu) = \frac{ D_{n',m'}^{(n)}( \vecv_n^{(\nu)}, \vecw_n^{(\nu)} ) }{ \sqrt{L_n n'} }, \qquad m', n' \ge 1,
\end{equation}
for $ 1 \le \nu \le L_n $. The corresponding martingale approximations are given by
\begin{equation}
\label{DefMartingalesScaled}
M_{n',m'}^{(n)}( \nu ) = \frac{ M_{n'+m'}^{(n)}( \vecv_n^{(\nu)}, \vecw_n^{(\nu)} ) 
	-  M_{m'}^{(n)}( \vecv_n^{(\nu)}, \vecw_n^{(\nu)} ) }{ \sqrt{L_n n'} },
\qquad m', n' \ge 1.
\end{equation}
We need to study the approximation error,
\[
R_{n',m'}^{(n)}( \nu ) = T_{n',m'}^{(n)}(\nu) - M_{n',m'}^{(n)}( \nu ), \qquad  n', m' \ge 1.
\]
The next result improves upon [Lemma~2]\cite{Kouritzin1995} by showing that, firstly, the error is of order $ (n')^{-\theta} $ in terms of the $ L_1 $--norm when conditioning on the past, and, secondly, that the result is uniform over $ \ell_1 $--bounded weighting vectors.

\begin{lemma} 
	\label{Lemma_Martingale_Approx}
	We have
	\[
	\left\| \sup_{1 \le \nu} E\left( ( R_{n',m'}^{(n)}( \nu ) )^2 \, | \, \calF_{m'} \right) \right\|_{1}
	\stackrel{n',m'}{\ll} L_n^{-1} (n')^{-\theta}. 
	\]
\end{lemma} 

\begin{proof}
	Consider, as in \cite{Kouritzin1995}, the decomposition
	\[
	R_{n',m'}^{(n)}( \nu ) = Q_{n',m'}^{(n)}( \nu ) + P_{n',m'}^{(n)}( \nu )  + O_{n',m'}^{(n)}( \nu ),
	\]
	where
	\begin{align}
	Q_{n',m'}^{(n)}( \nu ) & = \frac{1}{\sqrt{L_n n'}} \sum_{i=0}^{n'-1} \sum_{l=0}^{n'-i-1} 
	\wt{f}_{l,i+1}^{(n)}( \sigma_{m'-n'-i}^2 \eins_{\{l=0\}} - \epsilon_{m'+n'-i} \epsilon_{m'+n'-l} ), \\ 
	P_{n',m'}^{(n)}( \nu )  &= \frac{1}{\sqrt{L_n n'}}  \sum_{i=0}^\infty \sum_{l=0}^\infty 
	( \wt{f}_{l,i+1}^{(n)}( \nu ) - \wt{f}_{l,i+n'+1}^{(n)}(\nu) ) ( \epsilon_{m'-i} \epsilon_{m'-i-l} - \sigma_{m'-l}^2 \eins_{\{l=0\}} ),  \\
	O_{n',m'}^{(n)}( \nu )  &= - \frac{1}{\sqrt{L_n n'}} \sum_{i=0}^{n'-1} \sum_{k=n'}^\infty 
	\wt{f}_{k-i,i+1}^{(n)}( \nu )\epsilon_{m'+n'-k} \epsilon_{m'+n'-i}.
	\end{align}
	$ P_{n',m'}^{(n)}( \nu ) $ is the projection of $ R_{n',m'}^{(n)}( \nu ) $ onto the subspace spanned by 
	$ \{ \epsilon_r \epsilon_s - \sigma_r^2 \eins_{\{r=s\}} : - \infty < r, s \le m' \} $ and therefore $ \calF_{m'} $--measurable. Hence
	with $ e_{m',i,l} = \epsilon_{m'-i} \epsilon_{m'-i-l} - \sigma_{m'-l}^2 \eins_{\{l=0\}} $
	\begin{align*}
	L_n n' E\left( ( P_{n',m'}^{(n)}( \nu ) )^2 \, | \, \calF_{m'} \right) 
	& = \left( 
	\sum_{i=0}^\infty \sum_{l=0}^\infty 
	( \wt{f}_{l,i+1}^{(n)}( \nu ) - \wt{f}_{l,i+n'+1}^{(n)}(\nu) ) e_{m',i,l} 
	\right)^2 \\
	& \le 
	\left(
	\sum_{i=1}^\infty \sum_{l=0}^\infty | \wt{f}_{l,i+1}^{(n)}( \nu ) - \wt{f}_{l,i+n'+1}^{(n)}(\nu) | 
	| e_{m',i,l} | 
	\right)^2 \\
	& \le 
	\sum_{i=1}^\infty \sum_{l=0}^\infty ( \wt{f}_{l,i+1}^{(n)}( \nu ) - \wt{f}_{l,i+n'+1}^{(n)}(\nu) )^2 e_{m',i,l}^2 \\
	& \le 
	\sup_{k \ge 1} \gamma_k^2 \sup_{1 \le \nu} \sum_{i=0}^\infty \sum_{l=0}^\infty 
	( \wt{f}_{l,i+1}^{(n)}( \nu ) - \wt{f}_{l,i+n'+1}^{(n)}(\nu) )^2,
	\end{align*}
	a.s., such that due to (\ref{Ass1})
	\[
	E \sup_{1 \le \nu} E\left(  ( P_{n',m'}^{(n)}( \nu ) )^2 \, | \, \calF_{m'}  \right) 
	%      \le \sup_{k \ge 1} \gamma_k^2 \sup_{1 \le \nu} \sum_{i=0}^\infty \sum_{l=0}^\infty 
	%     ( \wt{f}_{l,i+1}^{(n)}( \nu ) - \wt{f}_{l,i'+n'+1}^{(n)}(\nu) )^2 
	\stackrel{m',n'}{<<} L_n^{-1} (n')^{-\theta}, a.s.
	\]
	$ Q_{n',m'}^{(n)}( \nu ) $ is the projection of $ R_{n',m'}^{(n)}( \nu ) $ onto the subspace spanned by 
	$ \{ \epsilon_r \epsilon_s - \sigma_r^2 \eins_{\{r=s\}} : m' < r, s \le m' + n' \} $ and thus independent from $ \calF_{m'} $, such that $ E\left( ( Q_{n',m'}^{(n)}( \nu ) )^2 \, | \, \calF_{m'} \right) = E ( Q_{n',m'}^{(n)}( \nu ) )^2 
	\stackrel{m',n'}{<<} \sup_{k \ge 1} \gamma_k (L_n n')^{-1} \sum_{i=1}^{n'} \sum_{l=0}^{n'-i} ( \wt{f}_{l,i}^{(n)}( \nu) )^2 \stackrel{m',n'}{<<} (n')^{-\theta} $,  uniformly in $ 1 \le \nu $. Last, by Fatou
	\begin{align*}
	& (L_n n') E \left( ( O_{n',m'}^{(n)}( \nu ) )^2 \, | \, \calF_{m'} \right) \\
	& \qquad \le \lim_{N \to \infty} \sum_{i,i'=0}^{n'-1} \sum_{k,k'=n'}^N \wt{f}_{k-i,i+1}^{(n)}( \nu )
	\wt{f}_{k-i',i'+1}^{(n)}( \nu )
	\epsilon_{m'+n'-k} \epsilon_{m'+n'-k'} E( \epsilon_{m'+n'-i} \epsilon_{m'+n'-i'} ) \\
	& \qquad \le \sup_{k \ge 1} \sigma_k^2 \lim_{N \to \infty}
	\sum_{i=0}^{n'-1} \sum_{k, k'=n'}^N | \wt{f}_{k-i,i+1}^{(n)}( \nu ) \wt{f}_{k'-i,i+1}^{(n)}( \nu ) \epsilon_{m'+n'-k} \epsilon_{m'+n'-k'} | \\
	& \qquad \le \sup_{k \ge 1} \sigma_k^2 \lim_{N \to \infty}
	\sum_{i=0}^{n'-1} \sqrt{ 
		\sum_{k, k'=n'}^N 
		[ \wt{f}_{k-i,i+1}^{(n)}( \nu ) ]^2 
		[ \wt{f}_{k'-i,i+1}^{(n)}( \nu ) ]^2
		\epsilon_{m'+n'-k}^2
		\epsilon_{m'+n'-k'}^2
	} \\
	& \qquad \le  \sup_{k \ge 1} \sigma_k^2
	\lim_{N \to \infty} \sum_{i=0}^{n'-1} \sum_{k=n'}^N [ \wt{f}_{k-i,i+1}^{(n)}( \nu ) ]^2 \epsilon_{m'+n'-k}^2
	\end{align*}  
	a.s., where we estimated the $ \ell_1 $--norm by the $ \ell_2 $--norm. Hence,
	\begin{align*}
	E \sup_{1 \le \nu} E \left( ( O_{n',m'}^{(n)}( \nu ) )^2  \, | \, \calF_{m'} \right) 
	& \le (L_n n')^{-1} \sup_{k \ge 1} ( \sigma_k^2 )^2 \sum_{i=0}^{n'-1} \sum_{l=1}^\infty \sup_{1 \le \nu} [ \wt{f}_{l,i}^{(n)}( \nu ) ]^2 \\
	& \stackrel{m',n'}{\ll}  L_n^{-1} (n')^{-\theta},
	\end{align*}
	by virtue of (\ref{Ass3}). This completes the proof. 
\end{proof}

\begin{proof}[Proof of Theorem~\ref{ThConvCovOp}]
	For a sequence of conditional covariance operators $ C_{n}( \cdot | \calA ) = E[ (\cdot, X_n) X_n \, | \, \calA ] $ with $ X_n = (X_{nj} )_j $, $ E(X_n) = 0$ , $n \ge 1 $, say, we have convergence in the operator semi-norm, defined as $ \| T \| = \sup_{f: \| f \| =1} | (f,Tf) | $ for an operator $T$ acting on $ \ell_2 $, to some unconditional covariance operator $ C(\cdot) = E[ (\cdot, Z) Z ] $, $ Z = ( Z_j )_j $, $ E(Z) = 0 $, in expectation, if 
	\[
	E \sup_{f \in \ell_2 : \| f \|=1} | (f, C_n(f|\calA) - C(f)) | = 
	E \sup_{f \in \ell_2 : \| f \|=1} \left| \sum_{j,k=1}^\infty  f_j f_k [ E(X_{nj} X_{nk}|\calA) - E(Z_jZ_k) ] \right| 
	\]
	converges to $0 $, as $ n \to \infty $. Define the $ \ell_2 $--valued random elements
	\[
	M_{n',m'}^{(n)} = \left( M_{n',m'}^{(n)}( \nu ) \right)_{\nu=1}^\infty, \qquad 
	T_{n',m'}^{(n)} = \left( T_{n',m'}^{(n)}( \nu ) \right)_{\nu=1}^\infty,
	\]
	where $ T_{n',m'}^{(n)}( \nu ) = 0 $ and $ M_{n',m'}^{(n)}( \nu ) = T_{n',m'}^{(n)}( \nu ) $, for $ \nu > L_n $. 
	Recall that
	\[
	\calC_{n',m'}^{(n)}(f | \calF_m ) = E[ (f, T_{n',m'}^{(n)} ) T_{n',m'}^{(n)} | \calF_{m'} ], 
	\qquad f \in \ell_2,
	\]
	and let
	\[
	\overline{\calC}_{n',m'}^{(n)}(f | \calF_m ) = E[ (f, M_{n',m'}^{(n)} ) M_{n',m'}^{(n)} | \calF_{m'} ], 
	\qquad f \in \ell_2,
	\]
	be the conditional covariance operator associated to the martingale approximations. Obviously,
	\[
	\sup_{n \ge 1}  E \| \calC_{n',m'}^{(n)}( \cdot | \calF_{m'} ) - \calC^{(n)}( \cdot ) \| \le
	\Xi_{n',m'} + \Psi_{n',m'}
	\]
	where
	\begin{align*}
	\Xi_{n',m'} & = \sup_{n \ge 1} E \| \calC_{n',m'}^{(n)}( \cdot | \calF_{m'} ) - \overline{\calC}^{(n)}_{n',m'}( \cdot | \calF_{m'} ) \|, \\
	\Psi_{n',m'} & = \sup_{n \ge 1} E \| \overline{\calC}^{(n)}_{n',m'}( \cdot | \calF_{m'} ) - \calC^{(n)}( \cdot ) \|.
	\end{align*}
	We shall estimate both terms separately. To simplify notation, let
	\begin{align*}
	C_{Tn}( \nu, \mu ) & = E\left( T_{n',m'}^{(n)}( \nu ) T_{n',m'}^{(n)}( \mu ) \, | \, \calF_{m'} \right), \\
	C_{Mn}( \nu, \mu ) & = E \left(  M_{n',m'}^{(n)}( \nu ) M_{n',m'}^{(n)}( \mu ) \, | \, \calF_{m'} \right), 
	\end{align*}
	for $ 1 \le \nu, \mu \le L_n $. To estimate $ \Xi_{n',m'} $, we shall show that $ | C_{Tn}( \nu, \mu ) - C_{Mn}( \nu, \mu ) | $ is $ O( L_n^{-1} (n')^{-\theta/2} ) $, uniformly in $ n ,\nu, \mu $. By an application of the Cauchy-Schwarz inequality, we have
	\begin{align*}
	& E \sup_{1 \le \nu, \mu} | C_{Tn}( \nu, \mu ) - C_{Mn}(\nu, \mu) | \\
	& = E \sup_{1 \le \nu, \mu}  \left| E \left(  T_{n',m'}^{(n)}( \nu ) T_{n',m'}^{(n)}( \mu ) \, | \, \calF_{m'} \right)
	-
	E\left(  M_{n',m'}^{(n)}( \nu ) M_{n',m'}^{(n)}( \mu ) \, | \, \calF_{m'}  \right)
	\right| \\
	& \le  E \sup_{1 \le \nu, \mu}  
	E\left(  | T_{n',m'}^{(n)}( \nu ) - M_{n',m'}^{(n)}(\nu) | | T_{n',m'}^{(n)}( \mu ) | \, \biggl| \, \calF_{m'} \right) \\
	& \quad 
	+ 
	E \sup_{1 \le \nu, \mu}  E\left( 
	| M_{n',m'}^{(n)}( \nu ) | |  T_{n',m'}^{(n)}( \mu )  - M_{n',m'}^{(n)}( \mu ) | \, \biggl| \, \calF_{m'} \right),
	\end{align*}
	where 
	\begin{align*}
	&  E \sup_{1 \le \nu, \mu}  E\left(  |  T_{n',m'}^{(n)}( \nu ) - M_{n',m'}^{(n)}(\nu) |  \bigl| T_{n',m'}^{(n)}( \mu ) |   \, | \, \calF_{m'} \right) \\
	& \qquad \le  E \sup_{1 \le \nu}  \sqrt{ E \left( ( T_{n',m'}^{(n)}( \nu ) - M_{n',m'}^{(n)}(\nu) )^2   \, | \, \calF_{m'} \right) }
	\sqrt{  \sup_{1 \le \nu} E | T_{n',m'}^{(n)}( \nu ) |^2  } \\
	& \qquad \stackrel{n',m'}{\ll} L_n^{-1} (n')^{- \theta/2},
	\end{align*}
	since $  T_{n',m'}^{(n)}( \mu ) $ is independent from $ \calF_{m'} $ and the decomposition
	$ T_{n',m'}^{(n)}( \mu ) = M_{n',m'}^{(n)}( \mu ) + ( T_{n',m'}^{(n)}( \mu ) - M_{n',m'}^{(n)}( \mu )) $ leads to
	$ \sup_{n \ge 1} \sup_{1 \le \nu} E | T_{n',m'}^{(n)}( \nu ) |^2 < \infty $, by virtue of (\ref{CovUncondBounded}) and
	Lemma~\ref{Lemma_Martingale_Approx}. Further,
	\begin{align*}
	& E\left(  | M_{n',m'}^{(n)}( \nu ) | |  T_{n',m'}^{(n)}( \mu )  - M_{n',m'}^{(n)}( \mu ) |    \, \bigl | \, \calF_{m'} \right) \\
	& \qquad \le \sqrt{ E \left( ( T_{n',m'}^{(n)}( \mu ) - M_{n',m'}^{(n)}(\mu)  )^2   \, | \, \calF_{m'} \right) }
	\sqrt{ E \left( | M_{n',m'}^{(n)}( \nu ) |^2 \, | \, \calF_{m'} \right) } \\
	& \qquad \stackrel{n',m'}{\ll} L_n^{-1} (n')^{- \theta/2},
	\end{align*}
	by (\ref{CovCondBounded}) and Lemma~\ref{Lemma_Martingale_Approx}.
	uniformly in $ \nu, \mu = 1, \dots, L_n $. Consequently,
	\[
	\sup_{n \ge 1} L_n  E \sup_{1 \le \nu, \mu} | C_{Tn}( \nu, \mu ) - C_{Mn}(\nu, \mu) | \stackrel{n',m'}{\ll}  (n')^{- \theta/2}
	\]
	Hence, using the inequality $ \sum_{\nu=1}^{L_n} | f_\nu | \le L_n^{1/2} \left( \sum_{\nu=1}^{L_n} f_\nu^2 \right)^{1/2} $, we obtain
	\begin{align*}
	\Xi_{n',m'} & \le 
	\sup_{n \ge 1}  E \| \calC_{n',m'}^{(n)}( \cdot | \calF_{m'} )
	- \overline{\calC}_{n',m'}^{(n)}( \cdot | \calF_{m'} ) \|  \\
	& =  \sup_{n \ge 1} E \sup_{f : \| f \|_{\ell_2} = 1}
	\left|
	\sum_{\nu=1}^{L_n} \sum_{\mu=1}^{L_n}  f_\nu  f_\mu 
	(C_{\nu \mu}^T - C_{\nu \mu}^M )
	\right| \\
	& \le 
	\sup_{n \ge 1} E  \sup_{1 \le \nu, \mu} | C^T_{\nu \mu} - C^M_{\nu\mu} | \sup_{f : \| f \|_{\ell_2} = 1} 
	\sum_{\nu=1}^{L_n} \sum_{\mu=1}^{L_n} | f_\nu f_\mu | \\
	& \stackrel{n',m'}{\ll} L_{n}^{-1} (n')^{-\theta/2} \left( \sum_{\nu=1}^{L_n} | f_\nu | \right)^2 %\\
	%& 
	\stackrel{n',m'}{\ll} (n')^{-\theta/2},
	\end{align*}
	By Lemma~\ref{HLemmaCov}, see (\ref{ConvCovCondF_m}), and the scaling of the martingale approximations, $ M_{n',m'}^{(n)}( \nu ) $, $ 1 \le \nu \le L_n $, by the factor $ (L_n n' )^{-1/2} $,
	\[
	\sup_{n \ge 1} L_n \max_{1 \le \nu, \mu \le L_n} \| E[ M_{n',m'}^{(n)}( \nu ) M_{n',m'}^{(n)}( \mu ) | \calF_{m'} ]
	- L_n^{-1} \beta_n^2( \nu, \mu ) \|_{L_1} \stackrel{n',m'}{\ll} (n')^{-\theta}.
	\]
	Therefore 
	\begin{align*}
	\Psi_{n',m'} & = 
	\sup_{n \ge 1} E \| \overline{\calC}_{n',m'}^{(n)}( \cdot | \calF_m )
	- \calC^{(n)}( \cdot ) \| \\
	& \le  \sup_{n \ge 1} E \sup_{\| f \|_{\ell_2} = 1} \left|
	\sum_{\nu=1}^{L_n} \sum_{\mu=1}^{L_n} 
	f_\nu f_\mu ( C_{\nu\mu}^M - L_n^{-1} \beta_n^2( \nu, \mu )) 
	\right| \\
	& \le \sup_{n \ge 1}  E \sup_{\| f \|_{\ell_2}=1}  \sup_{1 \le \nu, \mu} | C_{\nu\mu}^M - L_n^{-1} \beta_n^2( \nu, \mu ) |  
	\left( \sum_{\nu=1}^{L_n} | f_\nu | \right)^2  \\
	& \stackrel{n',m'}{\ll}  \sup_{n \ge 1}  L_n^{-1} (n')^{-\theta} \sup_{ \| f \|_{\ell_2} = 1 } L_n \sum_{\nu=1}^{L_n} f_\nu^2  %\\
	%& 
	\stackrel{n',m'}{\ll} (n')^{-\theta}.
	\end{align*}
\end{proof}

%We are now in a position to proof Theorem~\ref{Th_Many_QFs}.

\begin{proof}[Proof of Theorem~\ref{Th_Many_QFs}]
	By virtue of Lemma~\ref{LemmaRepr}, Equation (\ref{ReprSum}), we have the representations
	\begin{align*}
	D_{nt} =  \sum_{k \le t}  \left( \xi_k^{(n)}( j )  \right)_{j=1}^{L_n}, \quad
	\calD_n(t) 
	%       & = \left( L_n \calD_n(t; \vecv_n^{(j)}, \vecw_n^{(j)}) \right)_{j=1}^{L_n} \\
	%       & = \frac{1}{\sqrt{n L_n}} \sum_{k \le nt}  \left( \vecv_n^{(j)}{}' %[\wh{\bfSigma}_{nk} - 
	%     \bfSigma_{nk} ] \vecw_n^{(j)} \right)_{j=1}^{L_n} \\
	=  \frac{1}{\sqrt{n L_n}} \sum_{k \le \trunc{nt}}  \left( \xi_k^{(n)}( j )  \right)_{j=1}^{L_n},
	\end{align*}
	and therefore we check conditions (I) -- (III) of \cite{Philipp1986} discussed above
	for $ \zeta_k = L_n^{-1/2} \xi_k^{(n)} $, cf. (\ref{DefDLN}). The summands  can be seen as attaining values in the Euclidean space $ \R^{L_n} $ of finite (but increasing in $n$) dimension $ L_n $  or as random elements taking values in the infinite dimensional Hilbert space $ \ell_2 $.
	
	To show (I) observe that by the $C_r$--inequality, for each $ 1 \le j \le L_n $,
	\begin{align*}
	E | \xi_k^{(n)}(j) |^{2+\delta}
	& \le E ( | Y_{nk}( \vecv_n^{(j)} ) Y_{nk}( \vecw_n^{(j)} ) | + E | Y_{nk}( \vecv_n^{(j)} ) Y_{nk}( \vecw_n^{(j)} ) | )^{2+\delta} \\
	& \le 2^{3+\delta} E | Y_{nk}( \vecv_n^{(j)} ) Y_{nk}( \vecw_n^{(j)} ) |^{2+\delta},
	\end{align*}
	such that
	\begin{equation}
	\label{EstimateMomentsY}
	 E | \xi_k^{(n)}(j) |^{2+\delta} \le 2^{3+\delta} \sqrt{ E | Y_{nk}( \vecv_n^{(j)} ) |^{4+2\delta} } \sqrt{ E | Y_{nk}( \vecw_n^{(j)} ) |^{4+2\delta} }.
	\end{equation}
	Repeating the arguments of \cite[p.~343]{Kouritzin1995}, we obtain for $ \delta' \in (0,2) $ with $ \chi = \delta'/2 $
	\begin{align*}
	E | Y_{nk}( \vecv_n^{(j)} ) |^{4+\delta'}
	& \le \sup_{k' \ge 0} E | \epsilon_{k'} |^{4+\delta'} 
	\sum_{l=0}^\infty | c_{n\ell}^{(v_j)} |^{2(2+\chi)} \\
	& \quad + \sup_{k' \ge 0} E( \epsilon_{k'}^2 )
	\biggl\{
	\sup_{k' \ge 1} E( \epsilon_{k'}^2 )
	\biggr\}^{1+\chi} 
	\sum_{\ell=0}^\infty | c_{n\ell}^{(v_j)} |^2
	\biggl\{ \sum_{\ell'=0}^\infty | c_{n\ell}^{(v_j)} |^2  \biggr\}^{1+\chi},
	\end{align*} 
	uniformly in $ k \ge 1 $.
	But  $ \sup_{1 \le j} | c_{n\ell}^{(v_j)} | \le \| \vecv_n^{(j)} \|_{\ell_1} \sup_{1 \le \nu} | c_{n\ell}^{(\nu)} | $
	and   $ \sup_{1 \le \nu} | c_{n\ell}^{(\nu)} | \le (l \vee 1)^{-3/4-\theta/2} $,  due to Assumption (A), imply
	$ \sup_{1 \le j} \sum_{\ell=0}^\infty | c_{n\ell}^{(v_j)} |^2 < \infty $ and, in turn, 
	$ \sum_{l=0}^\infty | c_{n\ell}^{(v_j)} |^{2(2+\chi)} < \infty $. Noting that the above bounds hold uniformly in $ k $ and $n$, we obtain
	\[
	  \sup_{n \ge 1}  \sup_{k \ge 1} \max_{1 \le j \le L_n} E | \xi_k^{(n)}(j) |^{2+ \delta} < \infty.
	\]
	By virtue of Jensen's inequality, we may now conclude that
	\begin{align*}
	  E \| L_n^{-1/2} \xi_k^{(n)} \|_{\ell_2}^{2+\delta}
	    & = E \left( \frac{1}{L_n} \sum_{j=1}^{L_n} [ \xi_k^{(n)}( j ) ]^2 \right)^{1+\delta/2} 
	     \le L_n^{-1} \sum_{j=1}^{L_n} E | \xi_k^{(n)}(j) |^{2+\delta} 
	     < \infty,
	\end{align*}
	which establishes (I).  
	
	% We shall approximate the coordinates by the
	%  corresponding martingales $ \calM_n(t; \vecv_n^{(j)}, \vecw_n^{(j)} ) $, $ j = 1, \dots, L_n $,  and then  apply
	% the strong approximation obtained by \cite{Philipp1986} in general Hilbert spaces, observing that 
	%  the above process and the associated partial sums 
	Introduce the partial sums
	\begin{equation}
	\label{Snm}
	S_{n',m'}^{(n)} = \frac{1}{\sqrt{n' L_n}} \sum_{k=m'+1}^{m'+n'} \xi_k^{(n)}, \qquad n', m' \ge 1.
	\end{equation}
	Condition (II) can be shown as follows. 
	Denote the coordinates of $ S^{(n)}_{n',m'} $ by  $  S^{(n)}_{n',m'}(j) $ and notice that 
	they are given by $  S^{(n)}_{n',m'}(j) = \sum_{k=m'+1}^{m'+n'} L_n^{-1/2} \xi_k^{(n)}(j) $. Denote the
	corresponding martingale approximations by $  M^{(n)}_{n',m'} $ and $ M^{(n)}_{n',m'}(j) $, respectively,
	and let $ R^{(n)}_{n',m'} = S^{(n)}_{n',m'}  - M^{(n)}_{n',m'}  $ be the remainder 
	with coordinates $ R^{(n)}_{n',m'}(j) $, $ 1 \le j \le L_n $, cf. the preparations above. Clearly, the martingale property
	implies $ E( S^{(n)}_{n',m'}(j) | \calF_{m'} ) = E( R^{(n)}_{n',m'}(j) | \calF_{m'} ) $, $ 1 \le j \le L_n $.
	Lemma~\ref{Lemma_Martingale_Approx} asserts that
	\[
	\sup_{n \ge 1} L_n E \left[ \sup_{1 \le j}  E\left( ( R_{n',m'}^{(n)}(j) )^2 \, | \, \calF_{m'} \right) \right]
	\stackrel{n',m'}{\ll} (n')^{-\theta},
	\]
	such that two applications of Jensen's inequality lead to 
	\begin{align*}
	\sup_{n \ge 1} E \left| E ( S_{n',m'}^{(n)} \, | \, \calF_{m'} ) \right|_{\ell_2}
	& \le \sup_{n \ge 1}  \sqrt{  \sum_{j=1}^{L_n} E \left[   E( R_{n',m'}^{(n)}(j) \, | \, \calF_{m'} )  \right]^2 } \\
	& \le \sup_{n \ge 1}  \sqrt{  \sum_{j=1}^{L_n} E \left[ E \left( (R_{n',m'}^{(n)}(j))^2 \, | \, \calF_{m'} \right) \right] }\\
	& \le \sup_{n \ge 1} \sqrt{  L_n E \left[ \sup_{1 \le j} E \left( (R_{n',m'}^{(n)}(j))^2 \, | \, \calF_{m'} \right) \right] }\\
	& \stackrel{n',m'}{\ll} (n')^{-\theta/2},
	\end{align*}  
	which shows (II).   Condition (III) follows from Theorem~\ref{ThConvCovOp}. 
	
	Consequently, we may conclude that we may redefine all processes on a rich enough probability space where a Brownian motion $ B_n(t) = ( B_n(t)_j )_j $ with covariance operator $ \calC^{(n)} $, i.e. with covariances $ E ( B_n(t)_j B_n(t)_k ) = L_n^{-1} \beta_n^2(j,k) $, exists, such that for constants $ \lambda_n > 0 $ and $ C_n < \infty $
	\[
	\| L_n^{-1/2} D_{n,t}( \vecv_n, \vecw_n ) - B_n( t ) \|_{\R^{L_n}} 
	\le C_n t^{1/2-\lambda_n}, \qquad \text{for all $t\ge 0$},
	\]
	a.s.. Therefore
	\[
	\sup_{t \in [0,1]} \| \calD_n(t) - \overline{B}_n( \trunc{nt}/n ) \|_{\R^{L_n}} \le C_n n^{-\lambda_n}, 
	\] 
	a.s., for the $ [0,1] $--version $ \overline{B}_n $ of $B_n$, which implies assertions (i) and (ii). 
	
	% \textcolor{blue}{Die folgenden Aussagen beziehen sich alle auf (i). Sie sollten aber vermutlich umgeschrieben werden auf (ii), da die Approximation in $ \ell_2 $ auf jeden Fall ein $ o_P(1) $ liefert, auf Kosten des Faktors $ 1/\log \log(n) $. Sie deckt aber auf jeden Fall die Approximation der Partialsummen f"ur $ \delta_n \le k \le n $ ab. Aus $ \sup_{t \in [0,1] } $ wird $ \sup_{t \in [\delta_n/n, n/N]} $}.
	
	To show (iii) recall that the vector 1--norm of $ \R^{L_n} $ can be bounded by $ L_n^{-1/2} \| \cdot  \|_{\R_n^{L_n}} $, such that
	\[
	\sum_{j=1}^{L_n} | \calD_{nj}(t) - \overline{B}_n\left( \trunc{nt}/n \right)_j |
	\le L_n^{1/2} \| \calD_n(t) - \overline{B}_n( \trunc{nt}/n ) \|_{\ell_2} = o(L_n^{1/2}),
	\]
	as $ n \to \infty $, a.s. Further, using $ |x_j| %= \sqrt{x_j^2} 
	\le \sqrt{ \sum_j x_j^2 } $, we have
	\[
	\sup_{t \in [0,1]} | \calD_n(t; \vecv_n^{(j)}, \vecw_n^{(j)} ) - \overline{B}_n( \trunc{nt}/n )_j | \le \sup_{t \in [0,1]} \| \calD_n(t) - \overline{B}_n( \trunc{nt}/n ) \|_{\ell_2}
	= o(1),
	\]
	as $ n \to \infty $, a.s.,  for $ j = 1, \dots, L_n $.
		
	It remains to prove (iv).
	We may argue as in \cite{StelandvonSachs2016} to obtain
	\begin{align*}
	\vecv_n^{(\rho)}{}' ( \wh{\bfSigma}_{nmk} - \bfSigma_{nmk} ) \vecw_m^{(\sigma)}
	&=
	\sum_{i \le k} [ ( \vecv_n^{(\rho)}{}' \vecY_{ni} )( \vecw_m^{(\sigma)}{}' \vecY_{mi} ) - 
	E( ( \vecv_n^{(\rho)}{}' \vecY_{ni} )( \vecw_m^{(\sigma)}{}' \vecY_{mi} )  ) ] \\
	&= \sum_{i \le k} [ Y_{ni}( \vecv_n^{(\rho)} ){}' Y_{mi}( \vecw_m^{(\sigma)} ) - 
	E( Y_{ni}( \vecv_n^{(\rho)} ){}' Y_{mi}( \vecw_m^{(\sigma)} ) ) ]
	\end{align*}
	where 
	\[ 
	Y_{ni}( \vecv_n^{(\rho)} )  = \sum_{j=0}^\infty \sum_{\nu=1}^{d_n} v_{n\nu}^{(\rho)} c_{nj}^{(\nu)} \epsilon_{i-j} \quad \text{and} \quad
	Y_{mi}( \vecw_m^{(\sigma)} ) = \sum_{j=0}^{\infty} \sum_{\mu=1}^{d_n} w_{m\mu}^{(\sigma)} c_{mj}^{(\mu)} \epsilon_{i-j},
	\]
	for $ \rho = 1, \dots, L_n $ and  $ \sigma = 1, \dots, L_m $, $n, m \ge 1 $. Therefore, for $ k \ge 1 $, we obtain the representation
	\begin{align*}
	& D_k(  \{ (\vecv_n^{(\rho)}, \vecw_n^{(\rho)} ) \} ) \\
	& \qquad = \sum_{i \le k} \sum_{n,m} \sum_{\rho,\sigma} \lambda_n \lambda_m \mu_\rho \mu_\sigma
	[ \vecY_{ni}( \vecv_n^{(\rho)} ){}' \vecY_{mi}( \vecw_m^{(\sigma)} ) -
	E( \vecY_{ni}( \vecv_n^{(\rho)} ){}' \vecY_{mi}( \vecw_m^{(\sigma)} )  )  ] \\
	& \qquad = \sum_{i \le k} [Y_i( \{ c_j \} ) Y_i( \{ d_j \} ) - E( Y_i( \{ c_j \} ) Y_i( \{ d_j \} )  ) ]
	\end{align*}
	for the linear processes $ Y_i( \{ c_j \} ) = \sum_{j=0}^\infty c_j \epsilon_{i-j} $ and
	$ Y_i( \{ d_j \} ) = \sum_{j=0}^\infty d_j \epsilon_{i-j} $ with coefficients 
	\begin{align*}
	c_j & = \sum_{n=1}^\infty \lambda_n \sum_{\rho = 1}^{L_n} \mu_\rho \sum_{\nu=1}^{d_n} v_{n\nu}^{(\rho)} c_{nj}^{(\nu)}, \qquad j \ge 0 , \\
	d_j & = \sum_{n=1}^\infty \lambda_n \sum_{\rho = 1}^{L_n} \mu_\rho \sum_{\nu=1}^{d_n} w_{n\nu}^{(\rho)} c_{nj}^{(\nu)}, \qquad j \ge 0.
	\end{align*}
	Hence the result follows from \cite{Kouritzin1995}.
	%Lastly (perhaps omitted later...), The second assertion follows easily from the inequality
	% \begin{align*}
	%   | \Cov( X_n, Y_n ) - \Cov(X,Y) | 
	%   & \le E | X_n Y_n - XY | \\
	%   & \le E | X_n(Y_n - Y) + (X_n - X)Y | \\
	%  & \le \sqrt{E(X_n-X)^2} \sqrt{ EY_n^2 } + \sqrt{EX^2} \sqrt{E(Y_n-Y)^2}.
	% \end{align*}
\end{proof}

\begin{proof}[Proof of Remark~\ref{RemarkBBConstr}]
	By Theorem~\ref{Th_Basic},  we may and will assume that, on the same probability space, 
	\[
	| \calD_n(t; \vecv_n^{(j)}, \vecw_n^{(j)} ) - \alpha_n(\vecv_n^{(j)}, \vecw_n^{(j)} ) W_n( \trunc{nt}/{n}; \vecv_n^{(j)}, \vecw_n^{(j)} ) | = o(1),
	\]
	as $ n \to \infty $, a.s., for $L_n$ standard Brownian motions $ W_n( \trunc{nt}/{n}; \vecv_n^{(j)}, \vecw_n^{(j)} ) $, by virtue of Theorem~\ref{Th_Basic}. But then, since $ \calD_{nj}(t) = L_n^{-1/2} \calD_n( t; \vecv_n^{(j)}, \vecw_n^{(j)} ) $,
	\begin{align*}
	&  
	\sup_{t \in [0,1]} \left|  \overline{B}_n( \trunc{nt}/n )_j  - L_n^{-1/2} \alpha_n(\vecv_n^{(j)}, \vecw_n^{(j)} ) W_n( \trunc{nt}/{n}; \vecv_n^{(j)}, \vecw_n^{(j)} ) \right| \\
	&    
	\le \sup_{t \in [0,1]}   | \overline{B}_n( \trunc{nt}/n )_j  - \calD_{nj}(t) |  \\
	& + 
	\sup_{t \in [0,1]}   L_n^{-1/2} |  \calD_n( t; \vecv_n^{(j)}, \vecw_n^{(j)} ) - \alpha_n(\vecv_n^{(j)}, \vecw_n^{(j)} ) W_n( \trunc{nt}/{n}; \vecv_n^{(j)}, \vecw_n^{(j)} ) | \\
	& = o(1),
	\end{align*}   
	as $ n \to \infty $, a.s., for each $ j = 1, \dots, L_n $, which verifies the remark. 
\end{proof}

\begin{proof}[Proof of Theorem~\ref{WeakApprox_ManyQFs}]
	Observe that the conditions (I)-(III) of \cite[Theorem~1]{Philipp1986} hold in the Hilbert space $ \ell_2 $ as well, since for any $ \vecx \in \R ^{d_n} $ the Euclidean vector norm coincides with the $ \ell_2 $--norm. Therefore, we obtain the a.s. strong approximation
	\[
	\left\| \sum_{i=1}^k L_n^{-1/2} \xi_i^{(n)} - B_n(k) \right\|_{\ell_2} = \epsilon_{nk} \sqrt{ k \log \log k },
	\]
	as $ k \to \infty $, for sequences $ \varepsilon_{nk} = o(1) $,  $ k \to \infty $, a.s., $n \ge 1 $. Put $ N = \lceil n \log \log n \rceil $. Let $ \eta_n \downarrow 0 $ be given. Then for each $ n \in \N $ we may find $ \delta_n \in \N $ such that
	$ P( \max_{\delta_n \le k'} \epsilon_{nk'} > \varepsilon ) \le \eta_n $. Hence $ \max_{\delta_n \le k'} \epsilon_{nk'} = o_P(1) $, as $ n \to \infty $.
	Now we may  conclude that for  $ \delta_n \le k \le n $ 
	\[
	\left| \sum_{i=1}^k L_n^{-1/2} \xi_i^{(n)} - B_n(k) \right| \le \max_{\delta_n \le k' } \varepsilon_{nk'} \sqrt{N},
	\]
	such that
	\[
	\max_{\delta_n \le k \le n} \frac{1}{\sqrt{N}} \left|  \sum_{i=1}^k L_n^{-1/2} \xi_i^{(n)} - B_n(k)  \right| = o_P(1),
	\]
	as $ n \to \infty $, which verifies (\ref{StrongApprox_HSpace 1})--(\ref{StrongApprox_HSpace 3}). 	
\end{proof}

\section{Asymptotics for the trace norm}
\label{Sec Asymptotics Trace}

The trace plays an important role in multivariate analysis and also arises when studying shrinkage estimation. Before providing the large sample approximation by a Brownian motion, we shall briefly review its relation to several matrix norms. 

\subsection{The trace and related matrix norms}

There are various matrix norms that can be used to measure the size of (covariance) matrices. Here we shall use the trace norm defined as the
$ \ell_1 $-norm of the eigenvalues $ \lambda_i( \matA ) $ of a $d_n$-dimensional matrix $ \matA $, 
\[
 \| \matA \|_{tr} = \sum_i | \lambda_i( \matA ) |.
\]
Also notice that the trace norm is a linear mapping on the subspace of non-negative definite matrices and satisfies $ \| \matA \|_{tr} = \text{tr}( \matA ) $ for any covariance matrix $ \matA $. It induces the Frobenius norm via  $\| \matA \|_F^2 = \text{tr}(\matA \matA' ) $. Further, it is worth mentioning that the trace norm is also related to the Frobenius norm via the fact
\[
  \| \matA^{1/2} \|_F^2 = \sum_i \lambda_i( \matA ) = \| \matA \|_{tr}.
\]
In this way, our results formulated in terms of (scaled) trace norms can be interpreted in terms of (scaled) squared Frobenius norms of square roots, too.

There is a third interesting direct link to another family of norms, namely the Schatten-p norms $ \| \matA \|_{S,p} $, $ p \ge 0 $, of a $n \times m$ matrix $ \matA $ of rank $r$, which is defined as the $ \ell_p $-norm of its (non-negative) singular values
$ \sigma_1( \matA ) \ge \cdots \ge \sigma_r(\matA) > \sigma_{r+1}( \matA ) = \cdots = \sigma_n(\matA) = 0 $, i.e. of the eigenvalues of $ | \matA | = (\matA \matA')^{1/2} $, i.e.
\[
  \| \matA \|_{S,p}^p = \sum_i \sigma_i( \matA )^p.
\]
The Schatten-1 norm $ \| \matA \|_{S,1} $ is also called nuclear norm.
%Hence, in terms of those eigenvalues, the nuclear norm is the (convex) $ \ell_1 $-norm, %whereas the rank functional, i.e. the number of non-zero eigenvalues, is the  (non-convex) % $ \ell_0 $-norm. The nuclear norm represents the convex surrogate to the rank in the %sense that the nuclear unit
%is the smallest set containing all rank-1 matrices with spectral radius at most $1$
%(see e.g. Candez). 
Since $ \wh{\bfSigma}_n = \matA \matA' $, if $ \matA = \calY_n / \sqrt{n} $, such that $ \lambda_i( \wh{\bfSigma}_n ) = \sigma_i( \calY_n / \sqrt{n} )^2 $,  we have the identity
\[
  \| \wh{\bfSigma}_n \|_{tr} = \sum_i \lambda_i( \wh{\bfSigma}_n ) = \| \calY_n / \sqrt{n} \|_{S,2}^2.
\]
between the trace norm of the sample covariance matrix and the Schatten-$2$ norm of the scaled data matrix.

For a sequence $ \{ \matA_n \} $ of matrices of growing dimension $ d_n \times d_n $ it makes sense to attach a scalar weight depending on the dimension to a given norm, such that {\em simple} matrices such as the identity matrix receive bounded norms. Having in mind that the squared Frobenius norm of $ \matA_n $ is the trace of $ \matA_n \matA_n' $, it is natural to attach a scalar weight $ f(d_n) $ to the trace operator leading to the scalar weight $ f(d_n)^{1/2} $ for the Frobenius norm. As proposed by
\cite{LedoitWolf2004}, one may select $ f(d_n) $ such that 
$ \text{tr}( \matA^* ) f( d_n ) = 1 $ for some {\em simple} benchmark matrix $ \matA^* $ such as the $d_n$--dimensional identity matrix $ \matid_n $. Since $ \operatorname{tr}( \matid_{d_n} ) = d_n $, we choose $ f(d_n) = d_n^{-1} $ and therefore define the {\em scaled trace operator} by
\[
  \operatorname{tr}^*( \matA) = d_n^{-1} \sum_{i=1}^{d_n} \lambda_i( \matA )
\]
for a square matrix $ \matA = ( a_{ij} )_{i,j} $ of dimension $ d_n \times d_n $. The scaled trace operator induces the {\em scaled trace norm} 
\[
  \| \matA \|_{tr}^* = d_n^{-1} \| \matA \|_{tr}
\]
for a square matrix $ \matA $, which is given by  $ \| \matA \|_{tr}^* = d_n^{-1} \operatorname{tr}( \matA ) $ for a covariance matrix and averages the (modulus) of the eigenvalues, and the {\em scaled Frobenius matrix norm}  given by
\begin{equation}
\label{DefScaledFrobenius}
  \| \matA \|_F^{*2} = \operatorname{tr}^*( \matA \matA' ) = d_n^{-1} \| \matA \|_F^2.
\end{equation}

\subsection{Trace asymptotics}

Let us now turn to the trace asymptotics. If the dimension is fixed, it is well known that the eigenvalues of a sample covariance matrix, and thus their sum as well, have convergence rate $ O_P(n^{-1/2}) $ and are asymptotically normal,  see \cite{KolloNeudecker1993} and \cite{KolloNeudecker1994}. For the high-dimensional case, the situation is more involved. The sample covariance
matrix is not consistent w.r.t. to the Frobenius norm, even in the presence of
a dimension reducing factor model, see Remark~\ref{Remark_FactorModel}.

The following result provides a large sample normal approximation for the scaled trace norm of $ \wh{\bfSigma}_n $ for arbitrarily growing dimension $ d_n $ when properly normalized. The result also shows that the trace norm has convergence rate
\[
  \bigl| \|  \wh{\bfSigma}_n \|_{tr} - \| \bfSigma_n \|_{tr}  \bigr| = O_P( n^{-1/2} d_n ),
\]
as $ n \to \infty $. 

Introduce for $ t \in [0,1] $
\begin{equation}
\label{DefSigmaHat_t}
  \wh{\bfSigma}_n(t) = \frac{1}{n} \sum_{i=1}^{\trunc{nt}} \vecY_{ni} \vecY_{ni}', \qquad
  \bfSigma_n(t) = E( \wh{\bfSigma}_n(t)  ),
\end{equation}
and notice that
\[
  \wh{\bfSigma}_n = \wh{\bfSigma}_n(1) \qquad \text{and} \qquad \bfSigma_n = \bfSigma_n(1).
\]
We are interested in studying the scaled trace norm process
\begin{equation}
\label{ScaledTraceNormProces}
  \mathcal{T}_n(t) = \sqrt{n} \left( \| \wh{\bfSigma}_n(t) \|_{tr}^* - \| \bfSigma_n(t) \|_{tr}^* \right),
  \qquad t \in [0,1].
\end{equation}

\begin{theorem}
\label{Th_Trace} 
Let $ \vecY_{ni} $, $ i = 1, \dots, n $, be a vector time series following
model (\ref{ModelLinearProcess}) and satisfying Assumption (A). If (\ref{CrucialCond}) holds, then under the construction of
Theorem~\ref{Th_Many_QFs},
\[
  \sup_{t \in [0,1]} \left| \mathcal{T}_n(t) - d_n^{-1/2} \sum_{j=1}^{d_n} \overline{B}_{n}( \trunc{nt}/n )_j \right| = o(1),
\]
%where $ \alpha_n^{(i)} = \alpha_n( \vece_i, \vece_i ) $, $ i = 1, \dots, d_n$,
as $ n \to \infty $, a.s. Here $ \overline{B}_n $ denotes the $[0,1]$--version of the Brownian motion $B_n$ arising in Theorem~\ref{Th_Many_QFs}, when choosing the $ d_n $ pairs $ (\vecv_n^{(j)}, \vecw_n^{(j)}) = (\vece_j, \vece_j) $, $ j = 1, \dots, d_n $, where $ \vece_j $ denotes the $j$th unit vector, and satisfies properties (ii) and (iii) of Theorem~\ref{Th_Many_QFs}.
\end{theorem}

%In terms of the scaled trace norm $ \| \bullet \|_{tr}^{\#} $ we have
%\[
%  \left|
%  \sqrt{n d_n} \left(
%    \| \wh{\bfSigma}_n(t) \|_{tr}^{\#} - \| \bfSigma_n \|_{tr}^{\#}    
%  \right)
%  - d_n^{-1/2} \sum_{j=1}^{d_n} \overline{B}_{n}( \trunc{nt}/n )_j 
%  \right| = o(1),
%\]
%as $ n \to \infty $, a.s.

Suppose that, in addition to the assumptions of Theorem~\ref{Th_Trace}, $ \vecY_{n1}, \dots, \vecY_{nn} $ is strictly stationary.
Since the weighting vectors used in Theorem~\ref{Th_Trace} are the first $d_n$ unit vectors, the  covariance of $ \overline{B}_n(1)_i $ and $ \overline{B}_n(1)_j $, which is associated to the asymptotic covariance of $ \calD_{ni}(1)_i $ and $ \calD_{nj}(1)_j $, is given by 
$ d_n^{-1} \beta_n^2(i,j) $ where  $ \beta_n^2(i,j) = \beta_n^2( \vece_i, \vece_i, \vece_j, \vece_j ) $, $ i, j = 1, \dots, d_n $,
cf. (\ref{BetaNuMu}).
We have the asymptotic representations
\begin{align*}
  \beta_n^2(i, j)
  & =  \Cov\left( \sqrt{n} \vece_i'( \wh{\bfSigma}_n - \bfSigma_n) \vece_i, \sqrt{n} \vece_j'( \wh{\bfSigma}_n - \bfSigma_n) \vece_j \right) + o(1) \\
  & = \Cov\left( \frac{1}{\sqrt{n}} \sum_{k=1}^n [ ( Y_{nk}^{(i)} )^2 - \sigma^2 ], 
  \frac{1}{\sqrt{n}} \sum_{k=1}^n [ ( Y_{nk}^{(j)} )^2 - \sigma^2 ] \right)  + o(1) \\
  & = \frac{1}{n} \sum_{k,k'=1}^n E \left( ( Y_{nk}^{(i)} )^2 - \sigma^2 \right)\left( ( Y_{nk'}^{(j)} )^2 - \sigma^2 \right) + o(1).
  \end{align*}
Therefore, up to negligible terms, we may express $ \beta_n^2(i,j) $ as a  long--run variance parameter,
\begin{equation}
\label{DefLRVParameters}
  \beta_n^2( i, j )  = 
    \gamma_n^{(i,j)}( 0 ) + 2 \sum_{\tau=1}^{n-1} \frac{n-\tau}{n} \gamma_n^{(i,j)}( \tau ) + o(1),
\end{equation}
where
\begin{equation}
\label{DefCrossCovBeta}
  \gamma_n^{(i,j)}( \tau ) = \Cov\left( ( Y_{n0}^{(i)} )^2, ( Y_{n\tau}^{(j)} )^2 \right), \qquad \tau = 0, \dots, n-1,
\end{equation}
is the lag $ \tau $ cross--covariance of the series $ \{ ( Y_{nk}^{(i)} )^2 : k \ge 0 \} $ and $ \{ ( Y_{nk}^{(j)} )^2 : k \ge 0 \} $, $ i, j = 1, \dots, d_n $. Those cross-covariances can be estimated by
\[
  \wh{\gamma}_n^{(i,j)}( \tau ) = \frac{1}{n} \sum_{k=1}^{n-\tau}
  [ Y_k^2( \vece_i ) - \wh{\mu}_n(i) ]
  [ Y_{k+\tau}^2( \vece_j ) - \wh{\mu}_n(j)  ]
\]
with $ \wh{\mu}_n(i) = n^{-1} \sum_{k=1}^n Y_k^2( \vece_i ) $, where $ Y_k(\vece_i ) = \vece_i' \vecY_{nk} $, for $ k = 1, \dots, n $, $ i = 1, \dots, d_n $, $ n \ge 1 $.
The associated estimator for $ \beta_n^2(i,j) $ is then given by
\[
\wh\beta_{n}^2(i,j) = \wh\bfgamma_{n}^{(i,j)}(0) + 2 \sum_{\tau=1}^{m} w_{m\tau} \wh\bfgamma_{n}^{(i, j)}(\tau)\ ,
\]
where $m=m_n$ is a sequence of lag truncation constants and $\{w_{mh}\}$ a sequence of window weights typically defined by a kernel function $w$ (e.g. a Bartlett kernel) via $w_{m\tau}= w(\tau/b_m)$, for some bandwidth parameter $b=b_m$.

By Theorem~\ref{Th_Trace}, 
\[
  \Var( \calT_n(1) ) = \sigma_{tr,n}^2 + o(1), 
\]
a.s., with
\begin{equation}
\label{DefSigmaTrace}
\sigma_{tr,n}^2 =  \frac{1}{d_n} \sum_{j,k=1}^{d_n} \Cov( \overline{B}_n(1)_j, \overline{B}_n(1)_k ),
\end{equation}
and, using the canonical estimator
\begin{equation}
\label{DefEstimatorSigmaTrace}
  \wh{\sigma}_{tr,n}^2 = \frac{1}{d_n^2} \sum_{j,k = 1}^{d_n} \wh{\beta}_n^2(i,j),
\end{equation}
an asymptotic confidence interval with nominal coverage probability $ 1- \alpha $, $ \alpha \in (0,1) $, for $ \| \bfSigma_n \|_{tr}^* $ is given by
\[
  \| \wh{\bfSigma}_n \|_{tr}^* \pm \Phi^{-1}(1-\alpha/2) \wh{\sigma}_{tr,n} / \sqrt{n}.
\]
%Equivalently, an asymptotic confidence interval for $ \text{trace}( \bfSigma_n ) $ is
%\[
%  \text{tr}(  \wh{\bfSigma}_n ) \pm d_n \Phi^{-1}(1-\alpha/2) \wh{\eta}_n.
%\]

\begin{lemma} 
%	\label{ConsistencyEtaHat}
	\label{ConsistencyTraceVar}
	Assume $ w_{m\tau} \to 1 $ as $ m \to \infty $, for all $ \tau \in \Z $, and $ 0 \le w_{m\tau} \le W < \infty $, for some constant $W$, for all $ m \ge 1 $, $ \tau \in \Z $. Further,  suppose that $ c_{nj}^{(\nu)} = c_j^{(\nu)} $, $ \nu \in \N $, satisfy the decay condition
\[
  \sup_{1 \le \nu} | c_j^{(\nu)} | \ll (j \vee 1)^{-(1+\delta)}
\]
for some $ \delta > 0 $, and $ \epsilon_k  $ are i.i.d. with $ E( \epsilon_1^8 ) < \infty $. If $ m = m_n \to \infty $ with $ m^2/n = o(1) $, as $ n \to \infty $, then
\[
  \lim_{n \to \infty} E | \wh{\sigma}_{n,tr}^2 - \sigma_{n,tr}^2 | = 0.
\]
\end{lemma}

\begin{remark}
	\label{Remark_FactorModel} It is worth comparing our result with the following result obtained  by \cite{FanFanLv2008} for a dimension-reducing factor model: Suppose that the generic random vector $ \vecY_n = (Y_1, \dots, Y_{d_n})' $ satisfies a factor model
	\[
	\vecY_n = \matB_n \vecf + \boldsymbol{\epsilon}
	\] 
	with $ K = K(d_n) \le d_n $ observable factors $ \vecf = (f_1, \dots, f_K)' $,
	errors $ \boldsymbol{\epsilon} = ( \epsilon_1, \dots, \epsilon_{d_n})' $ and a $ d_n \times K $ factor loading matrix $ \matB_n $. Then 
	the sample covariance matrix of an i.i.d. sample $(\vecY_1, \vecf_1), \dots, (\vecY_n, \vecf_n) $
	has the convergence rate $ O_P( n^{-1/2} d_n K ) $,
	\[
	\| \wh{\bfSigma}_n - \bfSigma_n \|_F = O_P( n^{-1/2} d_n K ),
	\]
	if $ E \| \vecY_n \|_F^2 $, $ \max_i E( f_i^4 ) $ and $ \max_i E( \epsilon_i^4 ) $ are bounded,
	see \cite[Theorem~1]{FanFanLv2008}. This means, compared to the rate for fixed dimension, the Frobenius norm is inflated by the factor $d_n K$.  
\end{remark}

\subsection{Proofs}

\begin{proof}[Proof of Theorem~\ref{Th_Trace}]
	Clearly, $ \wh{\bfSigma}_k(1) $ is p.s.d. for all $ k \in \N $ and thus $ \wh{\bfSigma}_n(t) = \frac{ \trunc{nt} }{n} \wh{\bfSigma}_{\trunc{nt}} $ as well. The fact that $ \text{tr}(\matA) = \sum_i \vece_i' \matA \vece_i $ leads to
	\begin{align*}
	\| \wh{\bfSigma}_n(t) \|_{tr} - \| \bfSigma_n(t) \|_{tr} 
	%    &= \sum_{i=1}^{d_n} |\lambda_i(\wh{\bfSigma}_n) | - 
	% \sum_{i=1}^{d_n} | \lambda_i( \bfSigma_n ) | \\
	%    & = \sum_{i=1}^{d_n} \lambda_i(\wh{\bfSigma}_n) -
	% \sum_{i=1}^{d_n} \lambda_i( \bfSigma_n )  \\
	& = \text{tr}( \wh{\bfSigma}_n(t) ) - \text{tr}( \bfSigma_n(t) ) 
	%    & = \sum_{j=1}^{d_n} \vece_j' \wh{\bfSigma}_n \vece_j - \sum_i \vece_j' \bfSigma_n \vece_j 
	= \sum_{j=1}^{d_n} \vece_j'( \wh{\bfSigma}_n(t) - \bfSigma_n(t) ) \vece_j.
	\end{align*}
	Let $ \calD_n = ( \calD_{nj} )_{j=1}^{d_n} $ with $ \calD_{nj}(t) = d_n^{-1/2} \calD_n(t; \vece_j, \vece_j) $ for $ j = 1, \dots, d_n $. We shall apply Theorem~\ref{Th_Many_QFs} with $ L_n = d_n $. Therefore, when redefining all processes on a new probability space together with a $d_n$-dimensional Brownian motion $ \overline{B}_n $ with covariances as described in Theorem~\ref{Th_Many_QFs}, we may argue as follows.
	Since $ \| \cdot \|_{tr}^*  = d_n^{-1} \text{tr}( \cdot ) $, we have
	\[
	\mathcal{T}_n(t) = \frac{1}{\sqrt{d_n}} \sum_{j=1}^{d_n} \calD_{nj}(t).
	\]
	Now we  can conclude that the process
	\[
	\calE_n(t) = \sqrt{n}( \| \wh{\bfSigma}_n(t) \|_{tr}^* - \| \bfSigma_n(t) \|_{tr}^* )
	- \sum_{j=1}^{d_n} d_n^{-1/2} \overline{B}_n( \trunc{nt}/n )_j
	\]
	satisfies
	\begin{align*}
	| \calE_n(t) |
	& = \frac{1}{\sqrt{d_n}} \left| \sum_{j=1}^{d_n} \left( \calD_{nj}(t) - \overline{B}_n( \trunc{nt}/n )_j \right) \right| \\
	& \le \frac{1}{\sqrt{d_n}} \sum_{i=1}^{d_n} \left|  \calD_{nj}(t) - \overline{B}_n( \trunc{nt}/n )_j \right| \\
	%      & \frac{1}{\sqrt{d_n}} \sqrt{d_n} \| \calD_n(t) - \overline{B}_n( \trunc{nt}/n ) \|_{\ell_2}  \\
	%     &  = \sup_{t \in [0,1]} \| \calD_n(t) - \overline{B}_n( \trunc{nt}/n ) \|_{\ell_2} \\
	& = o(1),
	\end{align*}
	as $n \to \infty $, a.s., by Theorem~\ref{Th_Many_QFs} (iv).
\end{proof}

\begin{proof}[Proof of Lemma~\ref{ConsistencyTraceVar}]
The proof follows easily from \cite[Theorem~4.4]{StelandvonSachs2016} by noting that 
the covariances of the coordinates of the Brownian motion are given by $ d_n^{-1} \beta_n(i,j) $,
for $ 1 \le i, j \le d_n $.
\end{proof}

\section{Shrinkage estimation}
\label{Sec Asymptotics Shrinkage}

Shrinkage is a well-established approach to regularize the sample variance--covariance matrix and we shall review in Section~\ref{Subsec_Shrinkage} the results obtained for high--dimensional settings. When shrinking towards the identity matrix in terms of a convex combination with the sample variance--covariance matrix, the optimal weight depends on the trace of the true variance--covariance matrix, which can be estimated canonically by the trace of the sample variance--covariance matrix. As a consequence, we can apply the results obtained in the previous section to obtain large sample approximations for shrinkage variance--covariance matrix estimators. Recall that the approximations deal with the norm of the difference between partial sums and a Brownian motion, both attaining values in a vector space. In order to compare covariance matrices, we shall work with the following pseudometric: Define
\begin{equation}
\label{PseudoMetric}
  \Delta_n( \matA_n, \matB_n; \vecv_n, \vecw_n ) = | \vecv_n'( \matA_n - \matB_n ) \vecw_n |,
\end{equation}
for (sequences of) matrices $ \matA_n $ and $ \matB_n $ of dimension $ d_n \times d_n $. Indeed, for
fixed $ \vecv_n, \vecw_n $ the mapping $ ( \matA_n, \matB_n ) \mapsto \Delta_n( \matA_n , \matB_n ) = \Delta_n( \matA_n, \matB_n; \vecv_n, \vecw_n ) $ is symmetric, non-negative, semidefinite (i.e. $ \matA_n = \matB_n $ implies $ \Delta_n( \matA_n, \matB_n ) = 0 $) and satisfies the triangle inequality. Hence, (\ref{PseudoMetric}) defines a pseudometric on the space of $ (d_n \times d_n ) $--dimensional matrices, for each $n$.

We establish three main results: For {\em regular} weighting vectors $ \vecv_n, \vecw_n $ that are bounded away from orthogonality we establish a large sample approximation, which holds uniformly in the shrinkage weight and therefore also when using the common estimator for the optimal weight.  Further, we compare the shrinkage estimator using the estimated optimal weight with an oracle estimator using the unknown optimal weight. In both cases, it turns out that the convergence rate of the estimated optimal shrinkage weight carries over to the shrinkage covariance estimator.

Lastly, we study the case of orthogonal and {\em nearly orthogonal} vectors. The latter case  is of particular  interest, since then one may place more unit vectors on the unit sphere corresponding to overcomplete bases as studied in areas such as dictionary learning.

\subsection{Shrinkage of covariance matrix estimators}
\label{Subsec_Shrinkage}

The results of the previous chapters show that, under general conditions, inference relying on $ \ell_1 $-bounded inner products of high-dimensional series can be based on the sample covariance matrix, even if $ d_n > n $ such that $ \wh{\bfSigma}_n $ is singular. However, from a statistical point of view, the use of this classical estimator is not recommended in such situations of high dimensionality: important criteria such as its mean-squared error or its condition number (defined to be the ratio of the largest to the smallest eigenvalue) deteriorate, and it is advisable to regularise $ \wh{\bfSigma}_n $  in order to improve its performance, both asymptotically and for finite sample sizes, with respect to these criteria. Obviously, a particular interest lies in maximum-likelihood based approaches using an invertible estimator of  $ \bfSigma_n $ (as, e.g. in the semi-parametric approach of \cite{FiecasFrankeSachs2014} on shrinkage estimation in multivariate hidden Markov models). One well-established possibility to  regularise $ \wh{\bfSigma}_n $ without needing to impose any structural assumptions on  $\bfSigma_n $, in particular avoiding sparsity, is the following approach of shrinkage: (\cite{LedoitWolf2004}, \cite{Sancetta2008}) 
consider a shrinkage estimator defined by a linear (or convex) combination of  $ \wh{\bfSigma}_n $ with a well-conditioned "target" matrix ${\bf T}_n$,
\begin{equation}\label{eq:shrinkage_convec_comb}
\bfSigma_n^s =  \bfSigma_n^s(W_n) = (1-W_n) \wh{\bfSigma}_n\ +\ W_n\ {\bf T}_n\ ,
\end{equation}
where $W_n$ are the "shrinkage weights" of this convex combination, to be chosen in an optimal way to minimise the mean-square error between $ \bfSigma_n^s $ and $ \bfSigma_n $ (see below).
The role of the target ${\bf T}_n$ is, similar to ridge regression, to reduce a potentially large condition number of the high-dimensional variance-covariance matrix $ \bfSigma_n $, by adding a highly regular ("well conditioned") matrix. A popular choice for the target is to take a multiple of the $d_n-$dimensional identity matrix $\matid_n $, i.e.
\begin{equation}\label{eq:def_target}
{\bf T}_n = \mu_n\ \matid_n\ ,
\end{equation}
with $\mu_n = \frac{1}{d_n} \trace{\bfSigma}_n$, in order to respect the scale of both matrices in the convex combination (\ref{eq:shrinkage_convec_comb}). This choice of the target 
reduces the dispersion of the eigenvalues of $ \bfSigma_n $ around its "grand mean"  $\mu_n = \frac{1}{d_n} \trace{\bfSigma}_n$, as large eigenvalues are pulled down towards $\mu_n$ and small eigenvalues are lifted up to $\mu_n$ (and in particular lifted up away from zero). Although a bias is introduced in estimating $ \bfSigma_n $ by $  \bfSigma_n^s $ compared to $\wh{\bfSigma}_n\ $, the gain in variance reduction, in particular in high-dimensions, helps to considerably reduce the mean-square error in estimating $ \bfSigma_n $.

In order to develop the correct asymptotic framework of the behaviour of large covariance matrices, the authors of \cite{LedoitWolf2004} propose to use the {\em scaled Frobenius norm} given by (\ref{DefScaledFrobenius}) to measure the distance between two matrices of asymptotically growing dimension $d_n$, to be used also and in particular to define the mean-square error between $ \bfSigma_n^s $ and $ \bfSigma_n $  to become the expected normalised Frobenius loss $E[ \| \bfSigma_n^s - \bfSigma_n\|_F^{*2} ] $. 

Furthermore, with this scaling, 
\[
{\mu}_n = \frac{1}{d_n} \trace( {{\bfSigma}}_n) \ 
\]
is the appropriate choice of the factor in front of the identity matrix $ \matid_n $ in the definition of the target ${\bf T}_n $ in equation~(\ref{eq:def_target}).

In practice $\mu_n$ needs to be estimated from the trace of $ \wh{\bfSigma}_n $, i.e. by
\[
\wh{\mu}_n = \frac{1}{d_n} \trace({\wh{\bfSigma}}_n)\ .
\]
Similarly, the  theoretical shrinkage weight $0 \leq W_n \leq 1$ need to be replaced by its sample analog $\wh{W}_n$. Thus, the fully data-driven expression for the shrinkage estimator of $\bfSigma$ writes as follows:
\[
\bfSigma_n^s(\wh{W}_n)  = (1-\wh{W}_n) \wh{\bfSigma}_n\ +\ \wh{W}_n \wh{\mu}_n \matid_n,
\]
which shrinks the sample covariance matrix towards the (estimated) shrinkage target $ \wh{\mu}_n \matid_n $.
It remains to optimally choose the shrinkage weights $W_n$ (and its data-driven analogue  $\wh{W}_n$) with the purpose of balancing between a good fit and good regularisation.
For this a prominent possibility is indeed to choose the shrinkage weights $W_n$ such that the mean-squared error (MSE) between $ \bfSigma_n^s$ and $\bfSigma_n$,  is minimised:
%\[
%W_n^* = \text{argmin}_{W_n \in [0,1]} d_n^{-1} E[\| \bfSigma_n^s(W_n) - \bfSigma_n\|^2_F] \ ,
%\]
\[
W_n^* = \text{argmin}_{W_n \in [0,1]} E[\| \bfSigma_n^s(W_n) - \bfSigma_n\|_F^{*2}] \ ,
\]
which leads to the MSE-optimally shrunken matrix $\bfSigma_n^* =  \bfSigma_n^s (W_n^*) $.
A closed form solution (\cite{LedoitWolf2004} or \cite{Sancetta2008}, Proposition 1) can be derived as
\[
W^*_n = \frac{ E[\| \wh\bfSigma_n - \bfSigma_n\|_F^{*2}]}{ E[\| \mu_n \matid_n - \wh\bfSigma_n\|_F^{*2}]}
\]
This choice leads to the interesting property that
\[
E[\| \bfSigma_n^* - \bfSigma_n \|_F^{*2} <  E[\| \wh{\bfSigma}_n - \bfSigma_n \|_F^{*2} \ ,
\]
showing the actual relative gain of the shrunken estimator compared to the classical unshrunken sample covariance, in terms of the mean-squared error.
Moreover, it can be shown that this property continues to hold even if one replaces the in practice yet unknown optimal weights $W_n^*$ by an estimator $  \wh{W}^*_n $ which is constructed by replacing the population quantities in numerator and denominator of  $W_n^*$ by sample analogs. Whereas the denominator can be essentially estimated by  $\| \wh\mu_n \matid_n - \wh\bfSigma_n\|_F^{*2}$, it is slightly less straightforward to estimate the numerator $ E[\| \wh\bfSigma_n - \bfSigma_n\|^2_F$: one possibility suggested by {\cite{Sancetta2008}, and further developed by {\cite{StelandvonSachs2016} for our set-up, is based on the estimation of the 
		long-run variance $\alpha_n$ of $\wh\bfSigma_n$. 
		
Note that
		\[
		E \| \wh{\bfSigma}_n - \bfSigma_n \|_{F}^{*2} =
		\frac{1}{n d_n}
		\sum_{i, j  = 1}^{d_n} \Var( \sqrt{n} \wh{\bfSigma}_n(i,j) ), 
		\quad 
		\wh{\bfSigma}_n(i,j) = \frac{1}{n} \sum_{k=1}^n 
		[ Y_{nk}^{(i)} Y_{nk}^{(j)} - E( Y_{nk}^{(i)} Y_{nk}^{(j)}  )  ].
		\]
		where, under stationarity, for $ 1 \le i, j \le d_n $,
		\[   
		\sigma_n^2(i,j) = \Var( \sqrt{n} \wh{\bfSigma}_n(i,j) ) = \Gamma_n^{(i,j)}(0) + 2 \sum_{\tau=1}^{n-1} (n-\tau) \Gamma_n^{(i,j)}( \tau )
		\] 
		with $ \Gamma_n^{(i,j)}( \tau ) = \Cov( Y_{nk}^{(i)} Y_{nk}^{(j)}, Y_{n,k+\tau}^{(i)} Y_{n,k+\tau}^{(j)} ) $, $ \tau = 0, \dots, n-1 $, $ n \ge 1 $.
        A consistent estimator of the optimal weights $ W_n^*$ can now be obtained  as follows. Let
		\[
		\wh\Gamma_{n}^{(i,j)}(\tau) = \frac{1}{n} \sum_{t=1}^{n-\tau} [Y_{nt}^{(i)} Y_{nt}^{(j)} - \wh{\kappa}_{n}(i,j)]\  [Y_{n,t+\tau}^{(i)} Y_{n,t+\tau}^{(j)} - \wh{\kappa}_{n}(i,j)]\ ,
		\]
		where $ \wh{\kappa}_n(i, j) =  \frac{1}{n} \sum_{t=1}^{n} Y_{nt}^{(i)}Y_{nt}^{(j)} $, $ 1 \le i, j \le d_n $. Then,
		similar as in the previous section, the long-run variances $ \sigma_n^2(i,j) $ can be estimated by
		\[
		\wh\sigma_{n}^2(i,j) = \wh\Gamma_{n}^{(i,j)}(0) + 2 \sum_{\tau=1}^{m} w_{m\tau} \wh\Gamma_{n}^{(i, j)}(\tau)\ ,
		\]
		$ 1 \le i, j \le d_n $.
		Consistency of a more general version has been shown in \cite[Equation (4.22)]{StelandvonSachs2016}, under similar assumptions as stated in Lemma~\ref{ConsistencyTraceVar}, for $d_n \to \infty$, as $n\to\infty$. 
		We are led to the estimator	
		\[
		\wh{W}_n^* = \frac{(nd_n)^{-1} \sum_{1\leq i,j\leq d_n} \wh\sigma_n^2(i,j) }{\| \wh\mu_n \matid_n - \wh\bfSigma_n\|_F^{*2}}
		\]
		for $ W_n^* $ also studied in depth in \cite{Sancetta2008}.
		A rate of consistency in an asymptotic framework with growing dimensionality $d_n$ can be achieved again following \cite{Sancetta2008} for the specific shrinkage target $\mu_n \matid_n$, also considered in \cite{FiecasFrankeSachs2014}. Let $ 0 < \gamma \leq 2$ be such that $ d_n^{2-\gamma} / n \to 0 $ (i.e. the larger $\gamma$, the faster is $d_n$ allowed to grow with $n$), and that, as $n \to \infty$,
		\[
		d_n^{1-\gamma} \| \mu_n \matid_n - \bfSigma_n \|_F^{*2} \to c > 0\ .
		\]
		Recalling that  $ \mu_n = d_n^{-1} \trace( {\bfSigma_n} ), $ we observe that $ \gamma $ measures the closeness of the target to the true covariance matrix $\bfSigma_n$. Then \cite{Sancetta2008} (Theorem 1) and (\cite{FiecasFrankeSachs2014} (Theorem 1) show that 
		\begin{equation}
		\label{RvSRate}
		\frac{n}{d_n^{2-\gamma}} \ | \wh{W}_n^* - W_n^* | = o_P(1)\ .
		\end{equation}
		
		In order to apply the results of the previous sections onto the fully data-driven shrinkage estimator $\bfSigma_n^*(\wh{W}_n^*)$, one needs to study the convergence of the estimated shrinkage weight normalised by $n^{1/2}$, as will become clear from the proof of Theorem~\ref{Th_Shrinkage1} to be stated below.
		
		We already observe here that (\ref{RvSRate}) implies
		\begin{equation} 
		\label{CondShr}
		n^{1/2} | \wh{W}_n^* - W_n^* | = o_P(1), 
		\end{equation}
		if $ d^{4-2\gamma} / n = O(1) $. Thus, for $ \gamma $ close to $2$, the dimension
		$d_n $ may even grow faster than $ n $.

\subsection{Asymptotics for regular projections}

Our interest is now in deriving the asymptotics for bilinear forms based on the
shrinkage estimator of the covariance matrix. We can and will assume that the
uniformly $ \ell_1 $--bounded weighting vectors $ \vecv_n $ and $ \vecw_n $
are $ \ell_2 $--normed. It turns out that, due to the shrinkage target, the inner 
product $ \vecv_n'\vecw_n $, i.e. the angle between the vectors $  \vecv_n $ and $ \vecw_n  $, appears in the approximating
Brownian functional. The inner product is bounded but may converge to $0$, as $ n $ tends to $ \infty $. The latter case requires special treatment and will be studied separately. We shall call a pair $ ( \vecv_n, \vecw_n ) $ of projections {\em regular}, if it has uniformly bounded $ \ell_1 $--norm and satisfies
\begin{equation}
\label{RegularProj}
  \vecv_n'\vecw_n  \ge \underline{c} > 0, \qquad \text{for all $n$},
\end{equation}
for some constant $ \underline{c} $, i.e., if it is, in addition, bounded away from orthogonality. 

Let $ 0 < W \le 1 $ be an arbitrary shrinkage weight and consider the associated
shrinkage estimator
\[
  \wh{\bfSigma}_n^s(W) = (1-W) \wh{\bfSigma}_n + W \wh{\mu}_n \matid_n.
\]
Notice that $ \wh{\bfSigma}_n^s(W) $ estimates the unobservable shrunken variance matrix
\[
  \bfSigma_{n0}^s(W) = (1-W) \bfSigma_n + W \text{tr}(\bfSigma_n) d_n^{-1} \matid_n.
\]
Define for $ 0 < W \le 1 $
\begin{equation}
\label{DefAn}
\calA_n(W) = \sqrt{n}  \vecv_n' ( \wh{\bfSigma}_n^s(W) - \bfSigma_{n0}^s(W) ) \vecw_n.
\end{equation}

We shall apply the trace asymptotics obtained in Theorem~\ref{Th_Trace}, i.e. 
\[
  \sqrt{n} ( \| \wh{\bfSigma}_n(t) \|_{tr}^* - \| \bfSigma_n \|_{tr}^*(t) ) = d_n^{-1/2} \sum_{j=1}^{d_n} \overline{B}_n( \trunc{nt}/n )_j 
   + o(1),
\]
as $ n \to \infty $, a.s. The variance of the approximating linear functional of the $d_n$--dimensional Brownian motion is given by
\[
  \sigma_{tr,n}^2 = \frac{1}{d_n} \sum_{i,j=1}^{d_n} \Cov( \overline{B}_n(1)_i, \overline{B}_n(1)_j ) =
  \frac{1}{d_n^2} \sum_{i, j = 1}^{d_n} \beta_n^2(i, j),
\]
where $ \beta_n^2(i, j) $ are long--run--variance parameters, see (\ref{DefLRVParameters}).
Since, typically, long--run--variance parameters have positive limits, it is natural to assume that
\begin{equation}
\label{ConditionShrinkageAsymp}
  \inf_{n \ge 1} \sigma_{tr,n}^2 \ge \underline{\sigma}_{tr}^2 > 0.
\end{equation}

\begin{theorem}
\label{Th_Shrinkage1}
Let $ ( \vecv_n, \vecw_n ) $ be a regular pair of projections. 
 Under the assumptions of  Theorem~\ref{Th_Trace} and condition  (\ref{ConditionShrinkageAsymp}) there exists on a new probability space, which carries
 an equivalent version of the vector time series, a Brownian motion
 $ \overline{B}_n(t) = ( \overline{B}_n(t)_j )_{j=0}^{d_n} $ on $ [0,1] $, such that
 \begin{equation}
 \label{UniformApproxShr}
   \sup_{0 < W \le 1} | \calA_n(W) - \calB_n(W) | = o(1),
 \end{equation}
as  $ n \to \infty $, a.s.,  where
\begin{equation}
\label{DefApproxFunctional}
  \calB_n(W) = (1-W)  \overline{B}_n( 1 )_0
  + W \vecv_n' \vecw_n d_n^{-1/2} \sum_{j=1}^{d_n} \overline{B}_n(1)_j,
\end{equation}
The covariance structure of $ \overline{B}_n(t) $ is given by
\[
  \Var(  \overline{B}_n(1)_0 ) = (d_n+1)^{-1} \alpha_n^2( \vecv_n, \vecw_n ), \quad
  \Cov( \overline{B}_n(1)_0, \overline{B}_n(1)_i ) = (d_n+1)^{-1} \beta_n^2( \vecv_n, \vecw_n, \vece_i, \vece_i ), 
\]
for $ i = 1, \dots, d_n $, and 
\[
  \Cov( \overline{B}_n(1)_i, \overline{B}_n(1)_j ) = (d_n+1)^{-1} \beta_n^2( i, j ), 
\]
for $ 1 \le i, j \le d_n $. Especially, for any deterministic or random sequence of
 shrinkage weights $ W_n $ we have the large sample approximation for the
 corresponding shrinkage estimator
 \[
   | \sqrt{n} \vecv_n' (\wh{\bfSigma}_n^s( W_n ) - {\bfSigma}_{n0}^s(W_n))\vecw_n - \calB_n( W_n ) | = o(1), %O(n^{-\lambda} )
 \]
as $ n \to \infty $, a.s..
\end{theorem}

Notice that
\begin{align*}
\Var( \calB_n(W) ) & = \frac{(1-W)^2}{d_n+1} \alpha_n^2(\vecv_n, \vecw_n) 
+ \frac{W^2(\vecv_n'\vecw_n)^2}{d_n(d_n+1)} \sum_{i,j=1}^{d_n} \beta_n^2(i,j) \\
& \qquad 
+ 2 \frac{(1-W)W \vecv_n'\vecw_n}{\sqrt{d_n(d_n+1)}} \sum_{j=1}^{d_n} \beta_n^2(\vecv_n,\vecw_n, \vece_j, \vece_j) + o(1),
\end{align*}
as $ n \to \infty $.
Hence, under assumption (\ref{RegularProj}), the variance of the approximating Wiener process adressing the nonparametric part of the shrinkage estimator is of the order $ O( d_n^{-1} ) $, whereas the variance of the term approximating the target is of the order $ O(1) $. This is due to the fact that we need a ($d_n+1$)-dimensional Brownian motion (from which $d_n$ coordinates are used to approximate the estimated target). This requires to scale {\em all} coordinates $ (d_n+1)^{-1/2} $, cf. Theorem~\ref{Th_Many_QFs}. 

The following theorem resolves that issue by approximating the shrinkage estimator by two Brownian motions, one in dimension $1$ for the nonparametric part and one in dimension $d_n$ for the target. Those Brownian motions are constructed separately, such that, a priori, nothing can be said about their {\em exact} covariance structure. It turns out, however, that the covariances converge properly. We shall see that for this alternative construction the terms of the resulting variance-covariance decomposition are of the same order.

\begin{theorem}
	\label{Th_Shrinkage1b} Let $ ( \vecv_n, \vecw_n) $ be a regular pair of projections.
	Suppose that the underlying probability space $ (\Omega, \calF, P) $ is rich enough to carry, in addition to the vector time series $ \{ \vecY_{ni} : 1 \le i \le n, n \ge 1 \} $, a uniform random variable $ U_1 $. Then there exist, on  $ (\Omega, \calF, P) $, a univariate Brownian motion $ \{ \overline{B}_n'(t)_0 : t \in [0,1] \} $ with mean zero and
	\[
	  \Cov( \overline{B}_n'(s)_0, \overline{B}_n'(t)_0 ) = \min(s,t) \alpha_n^2( \vecv_n; \vecw_n ),
	\]
	for $ s, t \in [0,1] $, and a mean zero Brownian motion $ \{ ( \overline{B}_n'(t)_j )_{j=1}^{d_n} : t \in [0,1] \} $ in dimension $ d_n $ with covariance function
	\[
	  \Cov( \overline{B}_n'(s)_i, \overline{B}_n'(t)_j ) = \min(s,t) d_n^{-1} \beta_n^2( i, j ),
	\]
	for $ s, t \in [0,1] $ and $ 1 \le i, j \le d_n $, such that
	\[
	  \left|  \sqrt{n} \vecv_n'( \wh{\bfSigma}_n^s( W ) - \bfSigma_{n_0}^s( W ) ) \vecw_n
	  - \calB_n( W )   \right| = o(1),
	\]
	as $ n \to \infty $, a.s., 
	with $ \calB_n'( W ) = (1-W) \overline{B}_n'(1)_0 + W \vecv_n' \vecw_n d_n^{-1/2} \sum_{j=1}^{d_n} \overline{B}_n'(1)_j $. Further,
	\begin{equation}
	\max_{1 \le j \le d_n} \label{ConvergenceCovBBs}
	  | \Cov( \overline{B}_n'(1)_0, \overline{B}_n'(1)_j ) - d_n^{-1/2} \beta_n^2( \vecv_n, \vecw_n, \vece_j, \vece_j ) = o(1),
	\end{equation}
	as $ n \to \infty $. 
\end{theorem}

Observe that 
\begin{align*}
\Var( \calB_n'(W) ) & = (1-W)^2 \alpha_n^2(\vecv_n, \vecw_n) 
+ \frac{W^2(\vecv_n'\vecw_n)^2}{d_n^2} \sum_{i,j=1}^{d_n} \beta_n^2(i,j) \\
& \qquad 
+ 2 \frac{(1-W)W \vecv_n'\vecw_n}{d_n} \sum_{j=1}^{d_n} \beta_n^2(\vecv_n,\vecw_n, \vece_j, \vece_j) + o(1),
\end{align*}
as $ n \to \infty $, a.s., where all three terms are $ O(1) $.

The above result shows that the nonparametric part, namely the sample covariance matrix  $ \wh{\bfSigma}_n $, as well as the shrinkage target $ \wh{\mu}_n \matid_n $ contribute
to the asymptotics. In this sense, shrinking with respect to the chosen scaled norms 
provides us with a large sample approximation that mimics the finite sample situation.

\subsection{Comparisons with oracle estimators}

Recall that an oracle estimator is an estimator that depends on quantities unknown to us such as the optimal shrinkage weight $ W_n^* $. Of course, it is of interest to study the distance between the shrinkage estimator $ \wh{\bfSigma}_n^s( \wh{W}_n^* ) $ with estimated optimal weight and the associated oracle using $ W_n^* $. In particular, the question arises
how the rate of convergence (\ref{RvSRate}) affects the difference between the fully data adaptive estimator and an oracle.

The next theorem compares the shrinkage estimator 
$ \wh{\bfSigma}_n^s = \wh{\bfSigma}_n^s( \wh{W}_n^* ) $ that uses the estimated optimal shrinkage weight $ \wh{W}_n^* $ and the {\em oracle estimator} 
\[ 
  \wh{\bfSigma}_n^s( W_n^* )  = (1-W_n^*) \wh{\bfSigma}_n + W_n^* \wh{\mu}_n \matid_n,
\]
which shrinks the sample covariance matrix towards the target using the optimal shrinkage weight $ W_n^* $, in terms of the pseudometric $ \Delta_n( \cdot, \cdot; \vecv_n, \vecw_n ) $ and thus considers the quantity
\[
  \Delta_n( \wh{\bfSigma}_n^s( \wh{W}_n^* ), \wh{\bfSigma}_{n}^s( W_n^* ); \vecv_n, \vecw_n ) = 
  \bigl| 
   \vecv_n'( \wh{\bfSigma}_n^s( \wh{W}_n^* ) - \wh{\bfSigma}_n^s( W_n^* ) )\vecw_n
  \bigr|.
\]
The following result shows that even now the rate of convergence is equal to the
rate of convergence of the estimator $ \wh{W}_n^* $.

\begin{theorem}
	\label{Th_Shrinkage_Comparision}
  Under the assumptions of Theorem~\ref{Th_Shrinkage1} and the construction described there we have, on the new probability space, 
  \begin{align*}
   &  
%   \bigl| \vecv_n'( \wh{\bfSigma}_n^s( \wh{W}_n^* ) - \wh{\bfSigma}_n^s( W_n^* ) )\vecw_n \bigr| 
\Delta_n( \wh{\bfSigma}_n^s( \wh{W}_n^* ), \wh{\bfSigma}_{n}^s( W_n^* ); \vecv_n, \vecw_n ) \\
& \   = 
    |\wh{W}_n^* - W_n^*| \biggl( 
      O(1) + O( \overline{B}_n(1)_0 n^{-1/2} )
        + O\left( n^{-1/2} d_n^{-1/2} \sum_{i=1}^{d_n} \overline{B}_n(1)_i \right)
        + o( n^{-1/2} ),
    \biggr)
  \end{align*}
  as $ n \to \infty $, a.s..
\end{theorem}

The next result investigates the difference between the shrinkage estimator and the oracle type estimator
\[
  \bfSigma_{n0}^s( W_n^* ) = (1-W_n^*) \bfSigma_n + W_n^* \text{tr}( \bfSigma_n ) d_n^{-1} \matid_n
\]
using the oracle shrinkage weight and assuming knowledge of  $\bfSigma_n $, in terms of the pseudo-distance
\[
  \Delta_n( \wh{\bfSigma}_n^s( \wh{W}_n^* ), \bfSigma_{n0}^s( W_n^* ); \vecv_n, \vecw_n )
  = \bigl| \vecv_n'[ \wh{\bfSigma}_n^s( \wh{W}_n^* ) - \bfSigma_{n0}^s( W_n^* ) ] \vecw_n 
  \bigr|.
\]

\begin{theorem}
\label{Shrinkage_Oracle1} 
  Under the assumptions of Theorem~\ref{Th_Shrinkage1} and the construction described there we have, on the new probability space, 
%  \begin{equation}
%  \label{Cond_Shrinkage_Oracle1}
%    | \vecv_n' \bfSigma_n \vecw_n | = O(1).
 % \end{equation}
  \[
    \Delta_n( \wh{\bfSigma}_n^s( \wh{W}_n^* ), \bfSigma_{n0}^s( W_n^* ); \vecv_n, \vecw_n ) 
    = n^{-1/2} \calB_n( \wh{W}_n^* ) + (W_n^* - \wh{W}_n^*) O(1) 
    + o(n^{-1/2} ),
  \]
  as $ n \to \infty $, a.s..
\end{theorem}

The above result is remarkable in that it shows that it is optimal in the sense that $  \Delta_n( \wh{\bfSigma}_n^s( \wh{W}_n^* ), \bfSigma_{n0}^s( W_n^* ); \vecv_n, \vecw_n ) $ inherits the rate of convergence from the estimator $ \wh{W}_n^* $ of the optimal shrinkage weight $ W_n^* $, cf. (\ref{RvSRate}).

%\begin{remark} 
%Condition (\ref{Cond_Shrinkage_Oracle1}) constrains the size of the true covariance
%matrix under consideration, which is frequently measured in terms of the Frobenius norm:
%The $ O(1) $ bound imposed on the quadratic form $ \vecv_n' \bfSigma_n  \vecw_n $ implies that the
%spectral norm of $ \bfSigma_n $ is bounded, which in turn leads to the bound $ \| \bfSigma_n \|_F^2 = \sum_{i=1}^{d_n} \lambda_i^2(\bfSigma_n) = O(d_n) $ for the Frobenius norm. 
%\end{remark}

\subsection{Nearly orthogonal projections}

Let $ \vecv_{n}^{(i)} $, $ i = 1, \dots, L_n $, be unit vectors in $ \R^{d_n} $ on which we may  project $ \vecY_n $, e.g. in order to determine the best approximating direction. Recall that the true covariance between two projections $ \vecv_{n}^{(i)}{}' \vecY_n $ and $ \vecv_{n}^{(j)}{}' \vecY_n $ is  
\[ 
  \Cov( \vecv_{n}^{(i)}{}' \vecY_n, \vecv_{n}^{(j)}{}' \vecY_n) = \vecv_{n}^{(i)}{}' \bfSigma_n \vecv_{n}^{(j)},
\]
and the corresponding shrinkage estimator is
\[
  \wh{\Cov}(  \vecv_{n}^{(i)}{}' \vecY_n, \vecv_{n}^{(j)}{}' \vecY_n ) 
  =  \vecv_{n}^{(i)}{}' \wh{\bfSigma}_n^s( \wh{W}_n^* ) \vecv_{n}^{(j)}.
\]
Clearly, those covariances vanish for $ i \not= j $, if the $ \vecv_n^{(i)} $ are chosen as eigenvectors of $  \wh{\bfSigma}_n^s( \wh{W}_n^* ) $, as in a classical principal component analysis (PCA) applied to the shrinkage covariance matrix estimator. But when analyzing high--dimensional data it is common to rely on procedures such as sparse PCA, see \cite{WittenTibshirani2009}, which yield sparse principal components. Then
analyzing the covariances of the projections $ \vecv_{n}^{(i)}{}' \vecY_n $ is of interest. 

If $ \calO_{L_n} = \{ \vecv_{n}^{(j)} : 1 \le j \le L_n\} $ is an orthogonal system and $ L_n < d_n $, then $ \calO_{L_n} $ spans a $L_n$--dimensional subspace of $ \R^{d_n} $. Of course, there are at most $L_n$ orthogonal vectors, i.e. $L_n$ cannot be larger than $ d_n $. However, if one relaxes the orthogonality condition 
\[ 
  \vecv_{n}^{(i)}{}' \vecv_{n}^{(j)} = 0, \qquad  i \not= j,
\]
then one can place much more unit vectors in the Euclidean space $ \R^{d_n} $ in such a way that their pairwise angles are small. Indeed,
\cite{Tao2013} provides an elegant proof of the following Kabatjanskii-Levenstein bound.

\begin{theorem}
	\label{CheapKLBound}
%\begin{quote} 
  {\sc (Cheap version of the Kabatjanskii-Levenstein bound, Tao 2013)}.\\ 
  Let $ \vecx_1, \dots, \vecx_m $ be unit vectors in $ \R^{d_n} $ such that
  $ | \vecx_i' \vecx_j | \le A d_n^{-1/2} $ for some $ 1/2 \le A \le \frac12 \sqrt{d_n} $.
  Then we have $ m \le \bigl( \frac{C d_n}{A^2} \bigr)^{C A^2} $ for some universal
  constant $C$. 
%\end{quote}
\end{theorem}

Theorem~\ref{CheapKLBound} motivates to study the case of  {\em nearly orthogonal} weighting vectors, defined as a pair $ (\vecv_n, \vecw_n) \in \mathcal{W} \times \mathcal{W} $ satisfying 
\begin{equation}
\label{AlmostOrth}
  \vecv_n' \vecw_n = o(1), \qquad n \to \infty,
\end{equation}
Now the asymptotics of the shrinkage estimator is as follows:

\begin{theorem}
\label{Th_Shrinkage_Orth}
 Let $ \vecv_n $ and $ \vecw_n $ be unit vectors satisfying the
nearly orthogonal condition (\ref{AlmostOrth}) and
suppose that the conditions of Theorem~\ref{Th_Shrinkage1} hold. Then
\[
  | \sqrt{n} \vecv_n'( \wh{\Sigma}_n^s( W ) - \wh{\bfSigma}_{n0}^s( W ) ) \vecw_n 
  - \calB_n( W ) | = o_P(1),
\]
as $ n \to \infty $, a.s., where
\[
  \calB_n( W ) = (1-W) \overline{B}_n(1)_0 + W \vecv_n'\vecw_n d_n^{-1/2} \sum_{i=1}^{d_n}  \overline{B}_n(1)_i.
\]
\end{theorem}

Observe that for asymptotically orthogonal weighting vectors the term  $ W \vecv_n'\vecw_n d_n^{-1/2} \sum_{i=1}^{d_n}  \overline{B}_n(1)_i $ corresponding to the (parametric) shrinkage target is $ o_P(1) $ and thus vanishes asymptotically. In this situation, the nonparametric part dominates in large samples.

\subsection{Proofs}

\begin{proof}[Proof of Theorem~\ref{Th_Shrinkage1}]
	First notice that (\ref{RegularProj}) ensures that the second term in (\ref{DefApproxFunctional}) does not converge to $0$ in probability, since $ (\vecv_n, \vecw_n) $ is a regular projection and condition (\ref{ConditionShrinkageAsymp}) ensures that 
	the Gaussian random variable $ V_n = d_n^{-1/2} \sum_{j=1}^{d_n} \overline{B}_n(1)_j $ is not $ o_P(1) $, since
	$$ \inf_{n \ge 1} P( | V_n | > \delta ) \ge \inf_{n \ge 1} P( | V_n/\sigma_{tr,n} | > \delta / \underline{\sigma} ) > 0 $$
	for any $ \delta > 0 $. Hence $ W \vecv_n' \vecw_n V_n = o_P(1) $ iff.  $ \vecv_n' \vecw_n = o(1) $, which is excluded by (\ref{RegularProj}).
	
	We argue similarly as in the proof of Theorem~\ref{Th_Trace}: Put $ \calD_n = ( \calD_{nj} )_{j=0}^{d_n} $ with $ \calD_{n0}(t) = L_n^{-1/2} \calD_n( t; \vecv_n, \vecw_n ) $ and
	$ \calD_{nj}(t) = L_n^{-1/2} \calD_n(t, \vece_j, \vece_j ) $, $ j = 1, \dots, d_n $, where $ L_n = d_n+1 $. Since the weighting vectors are uniformly $ \ell_1 $-bounded, Theorem~\ref{Th_Many_QFs} yields, on a new probability space where a process equivalent to $ \calD_n $ can be defined and will be denoted again by $ \calD_n $, the existence of a Brownian motion $ \{ ( \overline{B}_n(t) )_{j=0}^{d_n} : t \ge 0 \} $ as characterized in Theorem~\ref{Th_Shrinkage1}, such that
	\[
	| \sqrt{n} ( \vecv_n' \wh{\bfSigma}_n \vecw_n - \vecv_n' \bfSigma_n \vecw_n )
	- \overline{B}_n( 1 )_0 | = o(1),
	\]
	and 
	\[
	\left| \sqrt{n} ( \| \wh{\bfSigma}_n \|_{tr}^*   - 
	\| \bfSigma_n \|_{tr}^* ) 
	- d_n^{-1/2} \sum_{i=1}^{d_n}  \overline{B}_n( 1 )_i \right|
	= o(1),
	\]
	as $ n \to \infty $, a.s. Using these results, $ \wh{\mu}_n = \| \wh{\bfSigma}_n \|_{tr} d_n^{-1} $ and the fact that  $ | \vecv_n' \vecw_n | \le \| \vecv_n \|_{\ell_1} \| \vecw_n \|_{\ell_1} = O(1) $, 
	we have for any $ 0 < W \le 1 $
	\begin{align*}
	& | \calA_n(W) - \calB_n(W) | \\
	& \quad = \biggl|
	\sqrt{n} \vecv_n' (1-W)( \wh{\bfSigma}_n - \bfSigma_n) \vecw_n
	- (1-W) \overline{B}_n( 1 )_0  \\
	& \qquad 
	+ W \sqrt{n} \vecv_n'( \| \wh{\bfSigma}_n \|_{tr}^* - \| \bfSigma_n \|_{tr}^* ) \vecw_n
	- W\vecv_n'\vecw_n d_n^{-1/2} \sum_{i=1}^{d_n} \overline{B}_n( 1 )_i          
	\biggr| \\
	& \quad \le (1-W) | \sqrt{n} \vecv_n'(\wh{\bfSigma}_n - \bfSigma_n) \vecw_n
	- \overline{B}_n( 1 )_0 | \\
	&  \qquad + 
	W \biggl|\sqrt{n} \vecv_n'\vecw_n( \| \wh{\bfSigma}_n^* \|_{tr}^* - \| \bfSigma_n \|_{tr}^* ) 
	- \vecv_n'\vecw_n d_n^{-1/2} \sum_{i=1}^{d_n}  \overline{B}_n( 1 )_i \biggr| \\
	& \quad \le (1-W) | \sqrt{n} \vecv_n'(\wh{\bfSigma}_n - \bfSigma_n) \vecw_n
	- \overline{B}_n( 1 )_0 | \\
	&  \qquad + 
	W | \vecv_n'\vecw_n | \biggl|\sqrt{n} ( \| \wh{\bfSigma}_n^* \|_{tr}^* - \| \bfSigma_n \|_{tr}^* ) 
	- d_n^{-1/2} \sum_{i=1}^{d_n} \overline{B}_n( 1 )_i \biggr| \\
	& \quad = o(1), %O(n^{-\lambda} ) + O( n^{-\lambda} ),
	\end{align*}
	as $ n \to \infty $, a.s., which shows (\ref{UniformApproxShr}).
\end{proof}

\begin{proof}[Proof of Theorem~\ref{Th_Shrinkage1b}]
By Theorem~\ref{Th_Basic} there exist, on a new probability space $(\Omega',\calF',P') $, an equivalent  process $ \{ \wt{\vecY}_{ni} : i \ge 1 \} \stackrel{d}{=} \{ \vecY_{ni} : i \ge 1 \} $ and a Brownian motion $ \{ \wt{B}_n(t)_0 : t \in [0,1] \} $ on $ [0,1] $, such that 
\begin{equation}
\label{ApproxCoord0}
 \sup_{t \in [0,1]} | \sqrt{n} \vecv_n'( \wt{\bfSigma}_n(t) - \bfSigma_n(t) ) \vecw_n - \wt{B}_n(t)_0 | = o(1),
\end{equation}
as $ n \to \infty $, $P'$-a.s., where $ \wt{\bfSigma}_n(t) = n^{-1} \sum_{i=1}^{\trunc{nt}} \wt{\vecY}_{ni} \wt{\vecY}_{ni}' $.
By Billingsley's lemma, \cite[Section~21, Lemma~2]{Billingsley1999}, there exist Brownian motions
$ \{ \overline{B}_n'(t)_0 : t \in [0,1] \}_n \stackrel{d}{=} \{ \wt{B}_n(t)_0 : t \in [0,1] \}_n $ defined on the original probability space $ (\Omega, \calF, P) $, such that
\begin{equation}
\label{SApproxSh Omega}
\sup_{t \in [0,1]} | \sqrt{n} \vecv_n'( \wh{\bfSigma}_n(t) - \bfSigma_n(t) ) \vecw_n - \overline{B}_n'(t)_0 | = o(1),
\end{equation}
as $ n \to \infty $, a.s. Indeed, recall that the infinite product of a complete and separabe metric space is complete and separable; in our case $ (D[0,1])^\infty $ equipped with the usual metric, see \cite[p.~241]{Billingsley1999}, induced by th Skorohod metric making $ D[0,1] $ separable and complete. Then apply \cite[Sec.~21, Lemma~1]{Billingsley1999} with $ \nu = \mathcal{L}( \{ \sqrt{n} \vecv_n'( \wt{\bfSigma}_n( \cdot ) - \bfSigma_n( \cdot ) ) \vecw_n, \{ \wt{B}_n( \cdot )_0 \}_{n \ge 1} ) $ and $ \sigma = \{ \sqrt{n} \vecv_n'( \wt{\bfSigma}_n( \cdot ) - \bfSigma_n( \cdot ) )\vecw_n \} $ to conclude the existence of $ \tau =: \{ \overline{B}_n'( \cdot )_0 \} $, a function of $ \sigma $ and $ U_1 $, such that (\ref{SApproxSh Omega}) holds, where the convergence w.r.t. the supnorm follows from the a.s. continuity of $ \overline{B}_n'( \cdot )_0 $.

Further, by Theorem~\ref{Th_Trace}, there exist, on a new probability space $ (\Omega'',\calF'',P'') $, an equivalent vector time series
$ \{ \check{\vecY}_{ni} : i \ge 1 \} \stackrel{d}{=} \{ \vecY_{ni} : i \ge 1 \} $, and Brownian motions $ \{ ( \wt{B}_n(t)_j )_{j=1}^{d_n} : t \in [0,1] \} $ on $ [0,1] $  in dimension $ d_n $ characterized as in the theorem, such that
\begin{equation}
\label{SApproxSh OmegaPrime}
  \left| \sqrt{n} ( \| \check{\bfSigma}_n(t) \|_{tr}^* - \| \bfSigma_n(t) \|_{tr}^* ) -  d_n^{-1/2} \sum_{j=1}^{d_n} \wt{B}_n(1)_j \right| = o(1),
\end{equation}
as $ n \to \infty $, $ P'' $-a.s., where $ \check{\bfSigma}_n(t) = n^{-1} \sum_{i=1}^{\trunc{nt}} \check{\vecY}_{ni} \check{\vecY}_{ni}' $. Again, an application of Billingsley's lemma shows the existence of Brownian motions $ \{ ( \overline{B}_n'(t)_j )_{j=1}^{d_n} : t \in [0,1] \}_n \stackrel{d}{=} \{ ( \wt{B}_n(t)_j )_{j=1}^{d_n} : t \in [0,1] \}_n $, such that 
\begin{equation}
\label{SApproxSh Tr}
  \left| \sqrt{n} ( \| \wh{\bfSigma}_n(t) \|_{tr}^* - \| \bfSigma_n(t) \|_{tr}^* ) -  d_n^{-1/2} \sum_{j=1}^{d_n} \overline{B}_n'(1)_j \right| = o(1),
\end{equation}
as $ n \to \infty $, a.s., i.e. on the original probability space. 

A priori, we have no information on the exact second order structure of the two Brownian motions, but $ \overline{B}_n'(1)_j $ is close to the associated process  $ \calD_{nj}(t) = T_{n,0}^{(n)}(j) $, cf. (\ref{DefDScaled}), and to the corresponding martingale approximation $ M_{\trunc{nt},0}^{(n)}(j) $ defined in (\ref{DefMartingalesScaled}), which allows us to study the convergence of the covariances
$ \Cov( \overline{B}_n'(1)_0, \overline{B}_n'(1)_j ) = ( \overline{B}_n'(1)_0, \overline{B}_n'(1)_j )_{L_2} $, $ j = 1, \dots, d_n $. First observe that
\[
  \max_{1 \le j \le d_n} \| \overline{B}_n'(1)_j - M_{n,0}^{(n)}(j)  \|_{L_2} = o(1),
\]
see \cite[Lemma~2]{Kouritzin1995}, \cite{StelandvonSachs2016}, and Lemma~\ref{HLemmaCov}, because of (\ref{ApproxCoord0}) and since   $ ( \overline{B}_n'(t)_j )_{j=1}^{d_n} $ satisfies
$ \sup_{t \in[0,1]} \sum_{j=1}^{d_n} | \overline{B}_n'(t)_j - \calD_{nj}(t) |^2 = o(1) $, 
as $ n \to \infty $, a.s., see Theorem~\ref{Th_Many_QFs} (ii). Also notice that $ \| \overline{B}_n'(1)_j \|_{L_2} $ and $ \| M_{n,0}^{(n)}( j ) \|_{L_2} $ are $ O(1) $, uniformly in $ j \ge 0 $. Now use the decomposition
\[
  (X, Y)_{L_2} = ( X', Y')_{L_2} + ( X', Y - Y' )_{L_2} + ( X - X', Y )_{L_2} 
\]
for $ X, Y, X', Y' \in L_2 $ to conclude that
\begin{align*}
  \Cov( \overline{B}_n'(1)_0, \overline{B}_n'(1)_j ) )
   & = ( M_{n,0}^{(n)}( 0 ), M_{n,0}^{(n)}( j ) )_{L_2} + ( M_{n,0}^{(n)}( 0 ), \overline{B}_n'(1)_j -  M_{n,0}^{(n)}( j ) )_{L_2} \\
   & \qquad + (( \overline{B}_n'(1)_0 - M_{n,0}^{(n)}( 0 ), M_{n,0}^{(n)}( j ) )_{L_2},
\end{align*}
where the last two terms are $ o(1) $, uniformly in $j$. E.g., 
\[
 \max_{1 \le j \le d_n} | ( \overline{B}_n'(1)_0 - M_{n,0}^{(n)}( 0 ), M_{n,0}^{(n)}( j ) )_{L_2}  | 
  \le \sup_{1 \le j} \| \overline{B}_n'(1)_0 - M_{n,0}^{(n)}( 0 ) \|_{L_2} \| M_{n,0}^{(n)}( j ) \|_{L_2}  = o(1),
\]
as $ n \to \infty $. Combining these estimates with Lemma~\ref{HLemmaCov} yields
\[
 \max_{1 \le j \le d_n} | \Cov( \overline{B}_n'(1)_0, \overline{B}_n'(1)_j ) 
  - \Cov( M_{n,0}^{(n)}( 0 ),  M_{n,}^{(n)}( j ) ) |  
 \ll  d_n^{-1} + o(1),
\] 
which establishes (\ref{ConvergenceCovBBs}), since the covariances of the approximating martingales equal $ d_n^{-1/2} \beta_n^2(\vecv_n, \vecw_n, \vece_j, \vece_j) + o(1) $, as $ n \to \infty $; the factor $ d_n^{-1/2} $ is due to the  additional scaling of $ \calD_{nj}(t) $ to approximate $ d_n $ bilinear forms by Theorem~\ref{Th_Many_QFs}.
\end{proof}

\begin{proof}[Proof of Theorem~\ref{Th_Shrinkage_Comparision}]
	Recall that, since $ \bfSigma_n = \matC_n \Lambda \matC_n' $ and (\ref{SuffCond}), the elements
	of $ \bfSigma_n $ are uniformly bounded (in $n$), such that $ | \vecv_n' \bfSigma_n \vecw_n | \le \| \vecv_n \|_{\ell_1} \| \vecw_n \|_{\ell_1} = O(1) $. This in turn implies that $ \lambda_{\max}( \bfSigma_n ) = O(1) $ and
	\begin{equation}
	\label{tr_Sigma_d_n_finite}
	\text{tr}( \Sigma_n ) d_n^{-1} = O(1). 
	\end{equation}
	Put
	\begin{equation}
	\label{DefRN}
	R_n( \wh{W}_n^*, W_n^* ) 
	= \vecv_n' \bfSigma_{n0}^s( \wh{W}_n^* ) \vecw_n - \vecv_n' \bfSigma_{n0}^s( W_n^* ) \vecw_n
	\end{equation}
	and notice that
	\[
	R_n( \wh{W}_n^*, W_n^* )  
	=  (W_n^*-\wh{W}_n^*) \vecv_n' \bfSigma_n \vecw_n + (\wh{W}_n^* - W_n^*) \text{tr}( \bfSigma_n ) d_n^{-1}  \vecv_n'\vecw_n.
	\]
	Using  (\ref{tr_Sigma_d_n_finite}) we obtain the bound
	\begin{equation}
	\label{Rn Rate}
	R_n( \wh{W}_n^*, W_n^* ) =  (\wh{W}_n^* - W_n^*) O(1).
	\end{equation}
	Observe that
	\[
	\wh{\bfSigma}_n^s( \wh{W}_n^* ) - \wh{\bfSigma}_n^s( W_n^* )
	= (W_n^* - \wh{W}_n^*) \wh{\bfSigma}_n + (\wh{W}_n^* - W_n^*) \text{tr}( \wh{\bfSigma}_n ) d_n^{-1} \matid_n
	\]
	is equal to the difference 
	\[ 
	\bfSigma_{n0}^s( \wh{W}_n^* ) - \bfSigma_{n0}^s( W_n^* ) 
	=
	(W_n^* - \wh{W}_n^*) \bfSigma_n + (\wh{W}_n^* - W_n^*) \text{tr}( \bfSigma_n ) d_n^{-1} \matid_n
	\]
	when replacing $ \bfSigma_n $ by  $ \wh{\bfSigma}_n $.  We have
	\begin{align*}
	& \wh{\bfSigma}_n^s( \wh{W}_n^* ) - \wh{\bfSigma}_n^s( W_n^* ) \\
	& \qquad = (W_n^* - \wh{W}_n^*) \bfSigma_n + (\wh{W}_n^* - W_n^*) \text{tr}( \bfSigma_n ) d_n^{-1} \matid_n \\
	& \qquad + (W_n^* - \wh{W}_n^*) (\wh{\bfSigma}_n - \bfSigma_n) 
	+ (\wh{W}_n^* - W_n^*)( \text{tr}( \wh{\bfSigma}_n ) - \text{tr}( \bfSigma_n ) ) d_n^{-1} \matid_n 
	\end{align*}
	Using (\ref{Rn Rate}), we therefore obtain for the associated bilinear form
	\begin{align*}
	& \vecv_n'( \wh{\bfSigma}_n( \wh{W}_n^* ) - \wh{\bfSigma}_n^s( W_n^* ) ) \vecw_n \\
	& \qquad = R_n( \wh{W}_n^*, W_n^* )  \\
	& \qquad \qquad + (W_n^* - \wh{W}_n^*)  \vecv_n'( \wh{\bfSigma}_n - \bfSigma_n ) \vecw_n 
	+ (W_n^* - \wh{W}_n^*) \vecv_n'\vecw_n ( \| \wh{\bfSigma}_n \|_{tr}^* - \| \bfSigma_n \|_{tr}^* ) \\
	& \qquad =  (\wh{W}_n^* - W_n^*) O(1) \\
	& \qquad \qquad +  (W_n^* -\wh{W}_n^*) \left( \overline{B}_n(1)_0 n^{-1/2} + o(n^{-1/2} ) \right) \\
	& \qquad \qquad +
	(\wh{W}_n^* - W_n^*) \left( n^{-1/2} d_n^{-1/2} \sum_{i=1}^{d_n} \overline{B}_n(1)_i
	+ o(n^{-1/2}) \right) \vecv_n'\vecw_n,
	\end{align*}
	as $ n \to \infty $, a.s..
\end{proof}

\begin{proof}[Proof of Theorem~\ref{Shrinkage_Oracle1}]
	Recall that $ \calA_n(W) = \sqrt{n} \vecv_n'( \wh{\bfSigma}_n^s(W) - \bfSigma_{n0}^s(W) ) \vecw_n $. 
	We have 
	\begin{align*}
	&     \vecv_n'( \wh{\bfSigma}_n^s( \wh{W}_n^* ) - \bfSigma_{n0}^s( W_n^* ) ) \vecw_n  \\
	& \qquad =   \vecv_n' (\wh{\bfSigma}_n^s( \wh{W}_n^* ) - \bfSigma_{n0}^s( \wh{W}_n^* ) )\vecw_n +
	\vecv_n'( \bfSigma_{n0}^s( \wh{W}_n^* ) - \bfSigma_{n0}^s( W_n^* ) )\vecw_n \\
	& \qquad =
	n^{-1/2} \calA_n( \wh{W}_n^* ) + R_n( \wh{W}_n^*, W_n^* ) 
	\end{align*}
	where again  $  R_n( \wh{W}_n^*, W_n^* ) =  (\wh{W}_n^* - W_n^*)  O(1) $.
	Further, using
	\[
	\sqrt{n} \vecv_n'( \wh{\bfSigma}_n^s( \wh{W}_n^* ) - \bfSigma_{n0}^s( \wh{W}_n^* ) ) \vecw_n
	= \calB_n( \wh{W}_n^* ) + o(1),
	\]
	a.s., we arrive at 
	\begin{align*}
	&     \vecv_n'( \wh{\bfSigma}_n^s( \wh{W}_n^* ) - \bfSigma_{n0}^s( W_n^* ) ) \vecw_n  \\
	& \qquad =
	n^{-1/2} \calA_n( \wh{W}_n^* ) + R_n( \wh{W}_n^*,W_n^* )  \\
	& \qquad = n^{-1/2} ( \calB_n( \wh{W}_n^* ) + o(1) )
	+ O(\wh{W}_n^* - W_n^*)  ,
	\end{align*}
	as $ n \to \infty $, a.s., which completes the proof.
\end{proof}

\begin{proof}[Proof of Theorem~\ref{Th_Shrinkage_Orth}]
	Let $ \calA_n(W) $ be defined as in (\ref{DefAn}). 
	Arguing as in the proof of Theorem~\ref{Th_Shrinkage1}, we obtain
	\begin{align*}
	| \calA_n(W) - \calB_n(W) |
	&\le (1-W) | \sqrt{n} \vecv_n'( \wh{\bfSigma}_n - \bfSigma_n) \vecw_n  - \alpha_n( \vecv_n, \vecw_n) \overline{B}_n(1)_0 |\\
	& \qquad + W \biggl| \vecv_n'\vecw_n \sqrt{n} ( \| \wh{\bfSigma}_n \|_{tr}^* - \| \bfSigma_n \|_{tr}^* ) - \vecv_n'\vecw_n d_n^{-1/2} \sum_{i=1}^{d_n} \overline{B}_n(1)_i \biggr|
	\end{align*} 
	The first summand is $ o(1) $, a.s., by Theorem~\ref{Th_Basic}.
	Under assumption (\ref{AlmostOrth}), the second term can be bounded by
	\begin{align*}
	R_n &= \biggl| \vecv_n' \vecw_n \sqrt{n}( \| \wh{\bfSigma}_n \|_{tr}^* - \| \wh{\bfSigma}_n \|_{tr}^* ) - \vecv_n'\vecw_n \sum_{i=1}^{d_n}  \overline{B}_n(1)_i \biggr| \\
	& = W | \vecv_n'\vecw_n | \biggl| \sqrt{n}   ( \| \wh{\bfSigma}_n \|_{tr}^* - \| \wh{\bfSigma}_n \|_{tr}^* ) 
	- d_n^{-1/2} \sum_{i=1}^{d_n}  \overline{B}_n(1)_i \biggr| \\
	& = o(1) ,
	\end{align*}
	as $n \to \infty $, a.s., which completes the proof.
\end{proof}

\appendix

\section{Notation and formulas}

We denote
\[
  \sigma_k^2 = E( \epsilon_k^2 ),  \gamma_k = E( \epsilon_k^4 ).
\]
%\textbf{Assumption (A)} The sequences $ \{ c_{nj}^{(\nu)} : j  \in \N_0 \} $ satisfy
%\begin{equation}
%\label{SuffCond}
 % \sup_{n \in \N} \max_{1 \le \nu \le d_n} | c_{nj}^{(\nu)} |^2 << j^{-3/2-\theta/2}
%\end{equation}
%for some $ 0 < \theta < 1/2 $. Here and in the sequel $ a_n << b_n $ stands for $ a_n = O(b_n) $.

%As discussed in \cite{StelandSachs2016}, Assumption (A) can be interpreted as 
%a sparsity assumption on the eigenvectors of $ \bfSigma_n $.

The approximating martingales used to obtain the strong approximations require to
control the following quantities. For the reader's convenience, we reproduce them here from \cite{StelandvonSachs2016}  as well as some related formulas and results. Let
\begin{equation}
  f_{0,j}^{(n)} = f_{0,j}^{(n)}( \vecv_n, \vecw_n ) 
  = \sum_{\nu, \mu = 1}^{d_n} v_\nu w_\mu c_j^{(\nu)} c_j^{(\mu)},  \qquad j = 0, 1, \dots,
\end{equation}
\begin{equation}
  f_{l,j}^{(n)} = f_{l,j}^{(n)}( \vecv_n, \vecw_n ) 
  = \sum_{\nu, \mu = 1}^{d_n} v_\nu w_\mu [ c_j^{(\nu)} c_{j+l}^{(\mu)} + c_j^{(\mu)} c_{j+l}^{(\nu)} ], \qquad l = 1, 2, \dots; \ j = 0, 1, \dots,
\end{equation}
and
\begin{equation}
  \wt{f}_{l,i}^{(n)} = \wt{f}_{l,i}^{(n)}( \vecv_n, \vecw_n )
  = \sum_{j=i}^{\infty} f_{l,j}^{(n)} 
  = \sum_{j=i}^\infty \sum_{\nu, \mu = 1}^{d_n} v_\nu w_\mu [c_j^{(\nu)} c_{j+l}^{(\mu)} + c_j^{(\mu)} c_{j+l}^{(\nu)} ], \qquad l, i  = 0, 1, \dots.
\end{equation}

\begin{lemma_def} Suppose that $ \vecv_n, \vecw_n $ have uniformly bounded $ \ell_1 $-norm in the sense of equation~(\ref{l1Condition}). Then Assumption (A) implies 
%\textbf{Assumption (M):} 
%Suppose that 
\begin{equation}
\label{Ass1}
  \sup_{n \in \N} \sum_{i=1}^\infty \sum_{l=0}^\infty (\wt{f}_{l,i}^{(n)} -\wt{f}_{l,i+n'}^{(n)}  )^2 \le C (n')^{1-\theta}, \qquad \text{for all $ n' = 1, 2, \dots $},
\end{equation}
\begin{equation}
\label{Ass2}
  \sup_{n \in \N} \sum_{k=1}^{n'} \sum_{r=0}^\infty (\wt{f}_{r+k,0}^{(n)} )^2 \le C (n')^{1-\theta}, \qquad \text{for all $ n' = 1, 2, \dots $},
\end{equation}
\begin{equation}
\label{Ass3}
  \sup_{n \in \N} \sum_{k=1}^{n'} \sum_{l=0}^\infty (\wt{f}_{l,k}^{(n)} )^2  \le C (n')^{1-\theta}, \qquad \text{for all $ n' = 1, 2, \dots $},
\end{equation}
and there exist
\begin{equation}
\label{AlphaN}
 \alpha_n^2 = \alpha_n^2( \vecv_n, \vecw_n ) \ge 0, \qquad n \ge 1, 
\end{equation}
such that
\begin{equation}
\label{Ass4}
  (\wt{f}^{(n)}_{00})^2 \sum_{j=1}^{n'} ( \gamma_{m'+j} - \sigma_{m'+j}^4 )
+ \sum_{j=1}^{n'} \sum_{l=1}^{j-1} ( \wt{f}_{j-l,0}^{(n)} )^2 \sigma_{m'+j}^2 \sigma_{m'+l}^2 - n' \alpha_n^2
\le C (n')^{1-\theta},
\end{equation}
for all $ n', m' = 0, 1, \cdots $. 

Further, if
$ \vecv_n, \vecw_n, \wt{\vecv}_n, \wt{\vecw}_n $, $ n \ge 1 $, have uniformly
bounded $ \ell_1 $-norms, then there exist
\begin{equation}
\label{BetaNuMu}
 \beta_n^2 = \beta_n^2( \vecv_n,  \vecw_n, \wt{\vecv}_n, \wt{\vecw}_n  ),
 \qquad n \ge 1,
\end{equation}
with 
\begin{align}
\label{Ass5}
\wt{f}_{0,0}^{(n)}( \vecv_n, \vecw_n )
  \wt{f}_{0,0}^{(n)}( \wt{\vecv}_n, \wt{\vecw}_n ) 
  \sum_{j=1}^{n'} ( \gamma_{m'+j} - \sigma^4_{m'+j} )
&   +
  \sum_{j=1}^{n'} \sum_{l=1}^{j-1} \wt{f}_{j-l,0}^{(n)}( \vecv_n, \vecw_n ) 
\wt{f}_{j-l,0}^{(n)}( \wt{\vecv}_n, \wt{\vecw}_n ) \sigma^2_{m'+j} \sigma^2_{m'+l} \\ \nonumber
& - n' \beta_n^2(  \vecv_n,  \vecw_n, \wt{\vecv}_n, \wt{\vecw}_n  ) 
\stackrel{n',m'}{<<} (n')^{1-\theta}.
\end{align}
%Assumption (M), i.e. (\ref{Ass1})--(\ref{Ass4}) and (\ref{Ass5}), hold true.
\end{lemma_def}

\section*{Acknowledgments}

Part of this work has been supported by a grant of the first author from Deutsche Forschungsgemeinschaft (DFG), grant No. STE 1034/11-1, which he gratefully acknowledges. Rainer von Sachs gratefully acknowledges funding by contract "Projet d'Actions de Recherche Concert\' ees'' No. 12/17-045 of the ,,Communaut\'e fran\c caise de Belgique'' and by IAP research network Grant P7/06 of the Belgian government (Belgian Science Policy).

%\bibliographystyle{agsm}
%*flatex input: [hdasymp-B-rev-31.bbl]

% flatex input end: [hdasymp-B-rev-31.bbl]
%FLATEX-REM:\bibliographystyle{plain}
%FLATEX-REM:\bibliography{lit}

%\begin{thebibliography}{99}
%  \bibitem{KOURITZIN} Kouritzin, ...
%\end{thebibliography}

\end{document}